\documentclass[twoside,11pt]{amsart}
\usepackage{graphicx}
\usepackage{epstopdf}
\usepackage{amsmath, amsfonts, amssymb, ifthen}

\usepackage{caption}
\usepackage{subcaption}
\usepackage[justification=centering]{caption}
\usepackage[colorlinks=true]{hyperref}

\newtheorem{theorem}{Theorem}[section]
\newtheorem{corollary}[theorem]{Corollary}
\newtheorem{main}{Main Theorem}
\newtheorem{lemma}[theorem]{Lemma}
\newtheorem{proposition}[theorem]{Proposition}

\theoremstyle{definition}
\newtheorem{definition}[theorem]{Definition}

\newtheorem*{notation}{Notation}

\DeclareMathOperator\Arg{Arg} 

\newcommand{\mandel}{\mathcal{M}}
\newcommand{\setU}{\mathbf{U}}
\newcommand{\ellipse}{\mathcal{E}}
\newcommand{\setWa}{\mathbf{W}_{n,c}}

\newcommand{\setWck}{\mathbf{W}_{n,a,k}}
\newcommand{\setS}{\mathbf{S}_{t}}
\newcommand{\setUp}{\mathcal{U}}
\newcommand{\setWp}{\mathcal{W}_{n,a}}

\newcommand{\newellipse}{\mathcal{L}}

\begin{document}

\title{The boundedness locus and baby Mandelbrot sets for some generalized Mc{M}ullen maps}

\author{Suzanne Boyd}

\address{Department of Mathematical Sciences,
University of Wisconsin Milwaukee,
PO Box 413\\
Milwaukee, Wisconsin 53201 USA\\
sboyd@uwm.edu}

\author{Alexander J.\ Mitchell}

\address{Physical Sciences and Mathematics Department,
Wayne State College,
1111 Main Street\\
Wayne, Nebraska 68787 USA\\
almitch1@wsc.edu}

\date{\today}


\begin{abstract}
In this paper we study rational functions of the form 

\noindent \mbox{$R_{n,a,c}(z) = z^n + \dfrac{a}{z^n} + c,$} 
with $n$ fixed and at least $3$, and hold either $a$ or $c$ fixed while the other varies. We locate some homeomorphic copies of the Mandelbrot set in the $c$-parameter plane for certain ranges of $a$, as well as in the $a$-plane for some $c$-ranges.

We use techniques first introduced by Douady and Hubbard in \cite{douhub} that were applied for the subfamily \mbox{$R_{n,a,0}$} by Devaney in \cite{dhalo}.  These techniques involve polynomial-like maps of degree two.
\end{abstract}

      \maketitle
      
      \markboth{\textsc{S. Boyd and A. Mitchell}}
  {\textit{Mandelbrot for generalized McMullen maps}}
   

\footnotetext[1]{010 MSC: Primary: 37F10; Secondary: 37F46. Keywords: Complex Dynamical Systems, Mandelbrot set, Polynomial-Like Map, Rational Map, Iteration}

\footnotetext[2]{We would like to thank Robert Devaney and Laura DeMarco for helpful conversations, and Brian Boyd for the computer program ``Dynamics Explorer" which generated all of the Mandelbrot and Julia images in this paper.}

\section{Introduction}
\label{sec:introduction}

As a simple starting example we consider the family of quadratic polynomials
$$
P_c(z) = z^2 + c,~c \in \mathbb{C}.
$$
We define the \textit{Fatou set} of $P_c$ in the typical way, as the set of values in the domain where the iterates of $P_c$ is a normal family in the sense of Montel. The \textit{Julia set} of $P_c$ is also defined the usual way as the complement to the Fatou set. The \textit{filled Julia set} is the union of the Julia set and the bounded Fatou components.

The \textit{Mandelbrot Set}, $\mandel$, is the set of $c$-values such that the critical orbit of $P_c$ is bounded, here that is the orbit of $0$.  Figure \ref{Mandelbrotone} (left) is the Mandelbrot set drawn in the $c$-parameter plane of $P_c$.  For other functions, the set of parameter values where at least one critical orbit is bounded will be called the \textit{boundedness locus}.

\begin{figure}[htbp]
\centering
\includegraphics[width=.45\textwidth,keepaspectratio]{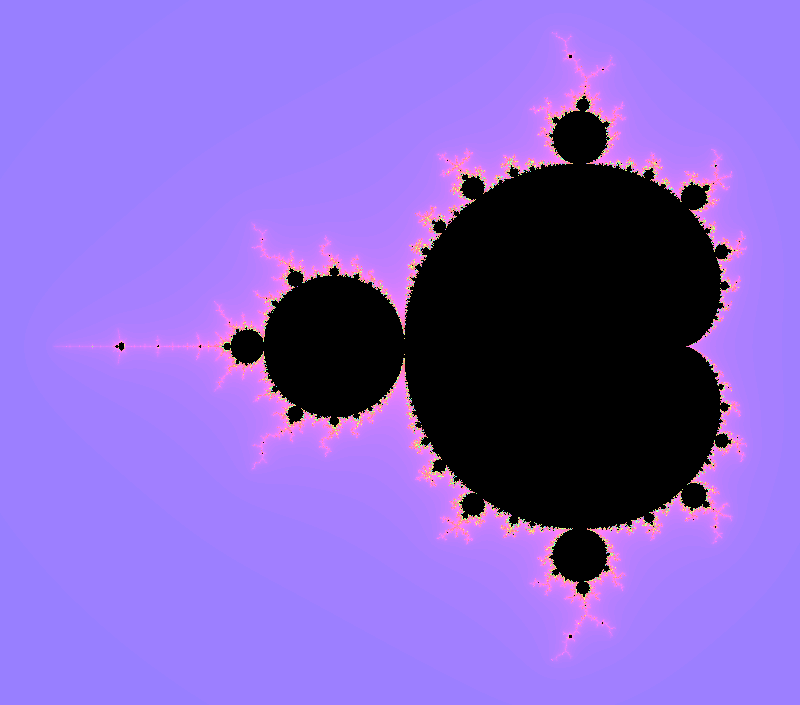}
\includegraphics[width=.45\textwidth,keepaspectratio]{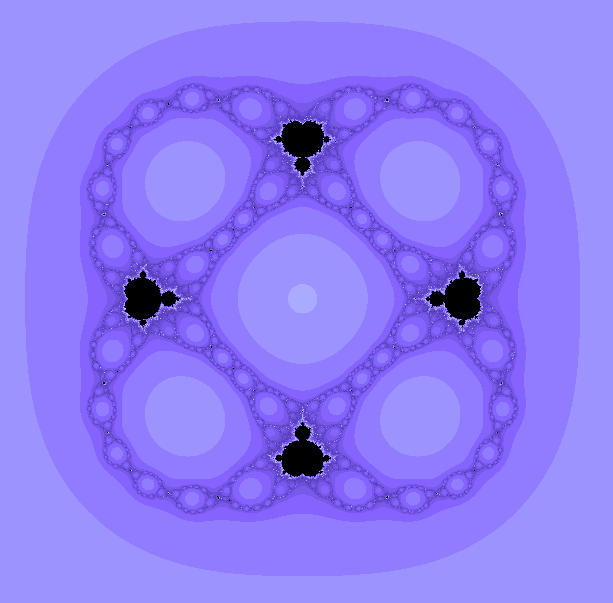}
\caption{Parameter Planes of $P_c$ (left) and $R_{5,a,0}$ (right).}
\label{Mandelbrotone}
\end{figure}

The study of the Mandelbrot set has become more accessible as computers have advanced.  Adrien Douady and John Hubbard were able to show that this set can result from other iterative processes as well, in \cite{douhub}.  They showed that multiple homeomorphic copies of the Mandelbrot set occur when Newton's Method is applied to a cubic polynomial family with a single parameter, and defined what it means for a map to behave like $P_c$, calling such a map {\em polynomial-like of degree two} (see Section \ref{preliminaries_section}). Mc{M}ullen (\cite{mcmullen}) shows that every non-empty bifurcation locus of any analytic family will contain quasiconformal copies of the Mandelbrot set of $P_c$ (or of $z^n + c$, based on the multiplicity of critical points), but in this paper we will use Douady and Hubbard's approach to prove that Mandelbrot set copies exist in some specific locations in some parameter planes, for the following family. 

The family of functions of interest in this paper is:
$$
R_{n,a,c}(z) = z^n + \dfrac{a}{z^n} + c~,~n \in \mathbb{N},~a \in \mathbb{C}\backslash \lbrace0\rbrace,~c \in \mathbb{C}.
$$
In this article, we restrict to integers $n \geq 3$.

This family, including the subfamily with $c=0$, has been studied previously by Robert Devaney and colleagues, as well as the first author and colleagues. In \cite{boydschul}, Boyd and Schulz study the geometric limit as $n\to \infty$ of Julia sets  and of the boundedness locus, for $R_{n,c,a}$ for any complex $c$ and any complex, non-zero $a$.  
Devaney and Garijo in \cite{devgar} study Julia sets as the parameter $a$ tends to $0$, for the cases of $n,d\geq 2$, and $c=0$. 
In \cite{bdgr} and \cite{devkoz}, the authors study the family in the case of $c$ at the center of a hyperbolic component of the Mandelbrot set for $P_c$ (that is, the critical point is a fixed point). 
 For $n\geq 2$, Devaney and colleagues study the subfamily with $c=0$, ``McMullen maps", in papers such as \cite{dhalo} and \cite{devsurv}. Our goal in this article is to generalize to the case $c\neq 0$ their result establishing the location of $n-1$ homeomorphic copies of the Mandelbrot set in the boundedness locus in the $a$-parameter plane of $R_{n,a,0}$ (see Figure \ref{Mandelbrotone} (right) for an example). 

In \cite{dhalo} and \cite{jangso} the authors find $n$ homeomorphic copies of Mandelbrot sets for a different generalization of McMullen Maps, $z \mapsto z^n + \dfrac{a}{z^d}$. 

We note that in \cite{xiaoqiu}, Xiao, Qiu, and Yongchen establish a topological description of the Julia sets (and Fatou components) of $R_{n,a,c}$ according to the dynamical behavior of the orbits of its free critical points. This work includes a result that if there is a critical component of the filled Julia set which is periodic, then the Julia set consists of infinitely many homeomorphic copies of a quadratic Julia set, and uncountably many points. In order to find baby Mandelbrot sets in our parameter planes of interest, we will first locate baby Julia sets, but using different techniques (based on specific parameter ranges rather than the type of dynamical behavior). 

Here, we consider the case where $c \neq 0$ (but $n=d$), and find homeomorphic copies of the Mandelbrot set in both the $a$ and $c$-parameter planes of $R_{n,a,c}$. Our main results are as follows.

\begin{main}
\label{Main_Theorem_APlane}
For the set of $n$ and $c$ values below, the boundedness locus in the $a$-parameter plane of $R_{n,a,c}$ contains a homeomorphic copy of the Mandelbrot set in the subset $\setWa$:
\begin{enumerate}
\item[(i)] $n\geq 3$ and $-1 \leq c \leq 0$;
\item[(ii)] odd $n \geq 3$ and $0 \leq c \leq 1$. 
\end{enumerate}
\end{main}

Item (i) is established in Theorem \ref{V+Mandel_APlane_Cnegative_theorem}, Item (ii) is shown in Corollary \ref{APlane_V-_Corollary}.   See Equation~\ref{eqn:defnW} for the definition of the set $\setWa$. 


\begin{main}
\label{Main_Theorem_CPlane}
For the set of $n$ and $a$ values below, the boundedness locus in the $c$-parameter plane of $R_{n,a,c}$ contains one or more homeomorphic copies of the Mandelbrot set, as follows. 
\begin{enumerate}
\item[(i)] For $n\geq 5$ and $1 \leq a \leq 4$, there are $n$ baby Mandelbrot sets, one in each subset $\setWck$ for $k \in \lbrace 0,1,...,n-1 \rbrace$; 
\\ if $n$ is odd there are at least $2n$, one within each $\setWck$ and one within its reflection over the imaginary axis;
\item[(ii)] For  $n \geq 11$ and $\frac{1}{10} \leq a \leq 1$, there is a baby Mandelbrot set in  $\setWp$;  if $n$ is odd there are at least two.
\end{enumerate}
\end{main}

Item (i) is established in Theorems~\ref{v+_Mandels_exist_multiple_cplane} and \ref{v-_Mandels_exist_multiple_through_symmetry_theorem}. 
See Equation~\ref{eqn:defnWck} for the definition of $\setWck$.

Item (ii) is shown in Theorem~\ref{Tighter_radius_mandel_exists_in_cPlane_Theorem} and 
Corollariy~\ref{v-_mandel_exists_in_cPlane_all_aValues_corollary}.
See Equation~\ref{eqn:defnWp} for the definition of $\setWp$.

To establish these results, we will take advantage of the many symmetries present in the family $R_{n,a,c}$. Our proof will follow the same general outline as in the case of $c=0$, but some additional complexities must be dealt with when $c\neq0$; for instance, there are multiple critical orbits to track.

We now discuss how the parameter planes of $R_{n,a,c}$ are drawn. With multiple critical orbits, it is more complicated than drawing $\mandel$ of $P_c$. To draw an $a$(or $c$)-parameter plane of $R_{n,a,c}$ we first fix a value of $n$ and $c$ (or $a$).  Then using each critical orbit, we color every point in the picture of the parameter plane as follows. 

First we assign a color (preferably unique) to each critical orbit.  Since there are two critical orbits here, $v_+$ and $v_-$, we assign green and purple, respectively.  For each parameter value in the picture we test both critical orbits for boundedness and assign a color for each orbit.  If the critical orbit is bounded we assign black. Else, if it escapes we assign that critical orbit's unique color, shaded based on rate of escape as is typical; that is, the shade of the color depends on the number of iterations it took for the orbit to escape a pre-defined escape radius.

Once the testing is complete, each parameter value has two RGB colors values assigned.  The computer will then \emph{average} the two values at each point, resulting in a single assigned color for that parameter value.

Therefore a parameter value with both critical orbits bounded will be colored black; a parameter with the critical orbit of $v_+$ bounded while $v_-$ escapes is colored dark purple and vice-versa is colored dark green; a parameter with both critical orbits escaping will be colored with the RGB average of the two colors. Note purple and green average to gray, and the colors only truly average if the rate of escapes match - if one escapes more slowly, that color is more intense, so it shades the grey toward purple or green.   Figures \ref{aPlane_Mandelbrot_Example_Figure} and \ref{cPlane_Multiple_mandelbrot_example} give examples of this coloring scheme used to draw the $a$ and $c$-parameter planes, respectively.

\begin{figure}
\centering
\includegraphics[width=.75\textwidth,keepaspectratio]{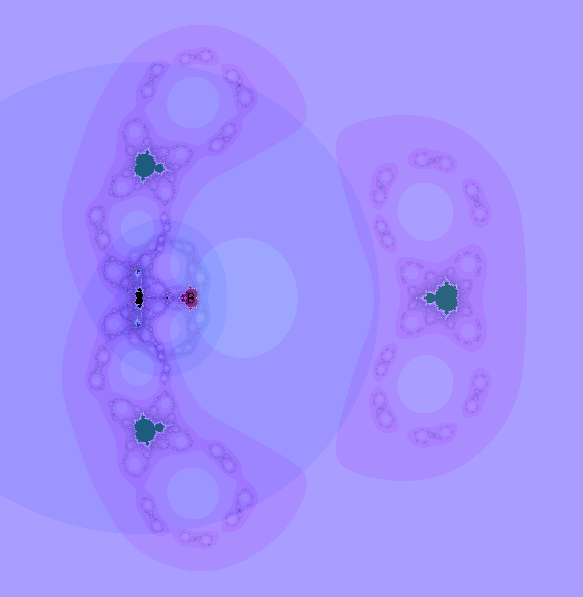} 
\caption{\label{aPlane_Mandelbrot_Example_Figure} The $a$-parameter plane of $R_{n,a,c}$ for $n=5$ and $c=0.5$. Each pixel is assigned two RGB values which are then averaged, one for each of the two critical values $v_+$ and $v_-$. If an orbit is bounded the color is black. If $v_+$ escapes the color assigned is green, for $v_-$ escaping it is purple. }
\end{figure}

\begin{figure}
\centering
\includegraphics[width=.75\textwidth,keepaspectratio]{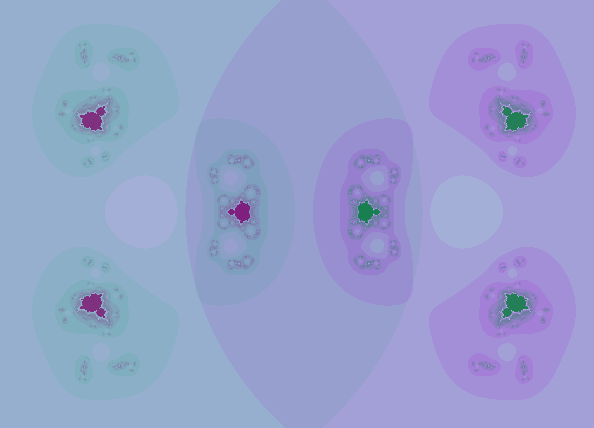} 
\caption{\label{cPlane_Multiple_mandelbrot_example} The $c$-parameter plane of $R_{n,a,c}$ for $n=3$ and $a=0.5$, colored in the same way as Figure~\ref{aPlane_Mandelbrot_Example_Figure}.}
\end{figure}

We close this introduction by previewing the organizaton of the sections. In Section \ref{preliminaries_section} we provide some background information, including Douady and Hubbard's criteria to prove existence of a Mandelbrot set in a region in parameter space, as well as some basic properties of the family $R_{n,a,c}$.  Section \ref{Main_One_Section} contains the main body of work needed to prove Main Theorem \ref{Main_Theorem_APlane}, in the $a$-plane. In Section \ref{Main_Two_Section} we turn to the $c$-plane and provide the proof of Main Theorem \ref{Main_Theorem_CPlane}-(i). 
 Finally, in section \ref{Extend_Results_Section} we remain in the $c$-plane but since $a=0$ is a degenerate case, we push toward results for smaller $a$-values - and prove Main Theorem \ref{Main_Theorem_CPlane}-(ii), and provide some additional results about situations in which baby Mandelbrot sets overlap.

\section{Preliminaries}
\label{preliminaries_section}

\begin{notation} The Mandelbrot set will be denoted throughout by $\mandel$, and we refer to a homeomorphic copy of $\mandel$ as a \textbf{baby $\mandel$}.
\end{notation}

To establish the existence of baby $\mandel$'s in a region in a parameter plane, we will use the definition of a polynomial-like map given by Douady and Hubbard:

\begin{definition}
\label{Polynomial-like_definition}
\cite{douhub} A map $F: \setU' \rightarrow F(\setU')=\setU$ is \textbf{polynomial-like} if
\begin{itemize}
\item $\setU'$ and $\setU$ are bounded, open, simply connected subsets of $\mathbb{C}$,
\item $\setU'$ relatively compact in $\setU$,
\item $F$ is analytic and proper.
\end{itemize}

Further $F$ is polynomial-like of \textbf{degree two} if $F$ is a $2$-to-$1$ map except at finitely many points, and $\setU'$ contains a unique critical point of $F$.

The \textbf{filled Julia set of a polynomial-like map} is the set of points whose orbits remain in $\setU'$:
 $\left\{z \in \setU' ~ \middle| ~ F^k(z) \in \setU', \forall k \in \mathbb{N} \right\}$.
\end{definition}

For a map satisfying this Definition \ref{Polynomial-like_definition}, Douady and Hubbard showed the following:

\begin{theorem}
\label{DH_Ploynomial-Like_theorem}
\cite{douhub}
A polynomial-like map of degree two is topologically conjugate on its filled Julia set to a quadratic polynomial on that polynomial's filled Julia set.
\end{theorem}

We will use this result later to locate homeomorphic copies of the filled Julia sets of $P_c$  in some particular dynamical planes of $R_{n,a,c}$.  

Douady and Hubbard provided criteria under which a family of polynomial-like functions possesses a baby $\mandel$ in a region $W$:

\begin{theorem}
\label{DH_Mandelbrot_existence_Criterion_Theorem}
Assume we are given a family of polynomial-like maps $F_\lambda: \setU'_\lambda \rightarrow \setU_\lambda$ that satisfies the following: 
\begin{itemize}
\item $\lambda$ is in an open set in $\mathbb{C}$ which contains a closed disk $W$;
\item The boundaries of $\setU'_\lambda$ and $\setU_\lambda$ vary analytically as $\lambda$ varies;
\item The map $(\lambda, z) \mapsto F_\lambda(z)$ depends analytically on both $\lambda$ and $z$;
\item Each $F_\lambda$ is polynomial-like of degree two with a unique critical point $c_\lambda$ in $\setU'$.
\end{itemize}

Suppose for all $\lambda \in \partial W$ that $F_\lambda(c_\lambda) \in \setU - \setU'$ and that $F_\lambda(c_\lambda)$ makes a closed loop around the outside of $\setU'$ as $\lambda$ winds once around $\partial W$.  

If all this occurs, then the set of $\lambda$-values for which the orbit of $c_\lambda$ does not escape from $\setU'$ is homeomorphic to the Mandelbrot set.
\end{theorem}

Theorem \ref{DH_Mandelbrot_existence_Criterion_Theorem} is key to establishing the location of some baby $\mandel$'s, like we see in Figures \ref{aPlane_Mandelbrot_Example_Figure} and \ref{cPlane_Multiple_mandelbrot_example}.

Because $\infty$ is a super-attracting fixed point of $R_{n,a,c}$, as it is for $P_c$, we can define
the {\bf filled Julia set of} $R_{n,a,c}$ as  the set of points whose orbits do not escape to $\infty$.

One thing to note about $R_{n,a,c}$ is that it has $2n$ critical points, $a^{\frac{1}{2n}}$. Though this could make it difficult to observe all critical orbits, it turns out that each of the critical points map to one of two values, $v_{\pm} = c \pm 2\sqrt{a}$.  Thus there are only two free critical orbits no matter the value of $n$.  Later we will study the effect of these two critical orbits.

We will exploit the following involution symmetry of $R_{n,a,c}$, to not only locate where the Julia set lies, but also to establish some cases in which $R_{n,a,c}$ is polynomial-like.

\begin{lemma}
\label{involution_prop}
$R_{n,c,a}$ is symmetric under the involution map $h_{a}(z)=\dfrac{a^\frac{1}{n}}{z}$.
\end{lemma}
\begin{proof}
$$
R_{n,c,a}(h_{a}(z))=\left(\frac{a^\frac{1}{n}}{z}\right)^n+\dfrac{a}{\left(\frac{a^\frac{1}{n}}{z}\right)^n}+c=\dfrac{a}{z^{n}}+z^{n}+c=R_{n,c,a}(z).
$$
\end{proof}

This symmetry will be used in both cases of the $a$ and $c$ parameter planes.

We will also use the following notation:

\begin{notation}
$\mathbb{D}(z_0,r)$ represents the disc $\left\{z ~ \middle|~ \lvert z - z_0 \rvert < r \right\}$.
\end{notation}

\begin{notation}
$\mathbb{A}(r,R)$ represents the annulus $\left\{z ~ \middle|~ r < \lvert z \rvert < R \right\}$.
\end{notation}


\section{The Case of $c$ fixed, $a$ varying}
\label{Main_One_Section}


In this section we establish Main Theorem~\ref{Main_Theorem_APlane}. Throughout, we will be under the following parameter restrictions:

\hangindent=0.7cm
\begin{itemize}
\item $n\geq 3$,
\smallskip
\item $\lvert c \rvert \leq 1$ and $c \in \mathbb{R}$,
\smallskip
\item $\dfrac{c^2}{4}\leq ~\lvert a \rvert~\leq \left( 1-\dfrac{c}{2}\right)^2$.

\end{itemize}

\subsection{Dynamical Plane Results}

Within these parameters we will restrict the location of the Julia set of $R_{n,a,c}$ (in Lemma \ref{EscapeAnnulus_lemma_aplaneCase}).  After that, we prove $R_{n,a,c}$ is polynomial-like of degree two (in Proposition ~\ref{R_Polynomial_like_on_first_U'_aplane_prop}).

First we take advantage of a result from \cite{boydschul}:

\begin{lemma}
\label{Julia_Set_Restriction_lemma}
\cite{boydschul}
For any $c \in \mathbb{C}$ and any $a \in \mathbb{C}$, given any $\epsilon > 0$, there is an $N \geq 2$ such that for all $n \geq N$ the filled Julia set of $R_{n,a,c}$ must lie in $\mathbb{D}(0,1+\epsilon)$, the disk of radius $1 + \epsilon$ centered at the origin.
\end{lemma}

This happens as the orbit of any point outside a radius of $1+\epsilon$ escapes to $\infty$, thus such a point with this behavior is not in the filled Julia set.  We apply this result to our case of restrictions on $n$, $a$, and $c$.

\begin{lemma}
\label{EscapeRadius_lemma_aplaneCase}
For $n \geq 3$, $\lvert c \rvert \leq 1$, and $\dfrac{c^2}{4}\leq ~\lvert a \rvert~\leq \left( 1-\dfrac{c}{2}\right)^2$, the filled Julia set of $R_{n,a,c}$ lies in the closed disk of radius $2$ centered at the origin.
\end{lemma}

\begin{proof}
The proof of Lemma \ref{Julia_Set_Restriction_lemma} in \cite{boydschul} says that if $N$ satisfies $(1+\epsilon)^N > 3 \text{Max} \lbrace 1,\lvert a \rvert, \lvert c \rvert \rbrace$ then for $n \geq N$ we have an escape radius of $1+\epsilon$. That is, the orbits of values $\lvert z \rvert > 1+\epsilon$ tend to $\infty$.  Setting $\epsilon=1$, by our constraints on $a$ and $c$, we have:
$$
3\text{Max} \lbrace 1,\lvert a \rvert, \lvert c \rvert \rbrace = 3\text{Max} \lbrace 1, \left( 1-\dfrac{c}{2}\right)^2, 1 \rbrace = 3(2.25) = 6.75
$$
for $a$ and $c$ at their greatest moduli. So when we solve this equation for $N$, we find $N > \dfrac{\ln(6.75)}{\ln(2)} \approx 2.75$, thus $n \geq 3$ will satisfy the criterion.  Therefore, the orbit of any $\lvert z \rvert > 2$ will escape to $\infty$ under iteration by $R_{n,a,c}$, hence the filled Julia set must lie in $\mathbb{D}(0,2)$.
\end{proof}

Combining this with Lemma \ref{involution_prop} restricts further the location of the filled Julia set of $R_{n,a,c}$.

\begin{lemma}
\label{EscapeAnnulus_lemma_aplaneCase}
With the same assumptions on $n$, $a$, and $c$ as Lemma \ref{EscapeRadius_lemma_aplaneCase}, the filled Julia set of $R_{n,a,c}$ lies within the annulus $\mathbb{A}\left(\dfrac{\lvert a \rvert^{\frac{1}{n}}}{2} ,~2\right)$.
\end{lemma}

\begin{proof}
Given any $\lvert z \rvert \leq \dfrac{\lvert a \rvert^{\frac{1}{n}}}{2}$ and the involution symmetry of Lemma \ref{involution_prop}, then
$\lvert R_{n,a,c}(z) \rvert \geq 2$ and thus the orbits of these values also escape to $\infty$.  Therefore the filled Julia is a subset of $\mathbb{A} \left( \dfrac{\lvert a \rvert^{\frac{1}{n}}}{2}, 2 \right)$.
\end{proof}

Figure \ref{Julia_set_in_Annulus_Figure} shows a Julia set of $R_{n,a,c}$ lying in this annulus.  We see various black shapes appearing in this dynamical plane and will actually prove below that these shapes are homeomorphic copies of a filled Julia set of a $P_{c}$. (The one in the figure appears to be a baby basilica $K_{-1}$ for which the critical value lies in a period two cycle). This occurs because $R_{n,a,c}$ is polynomial-like of degree two on those regions, which we prove in Proposition \ref{R_Polynomial_like_on_first_U'_aplane_prop}.

\begin{figure}
\centering
\includegraphics[width=0.75\textwidth, keepaspectratio]{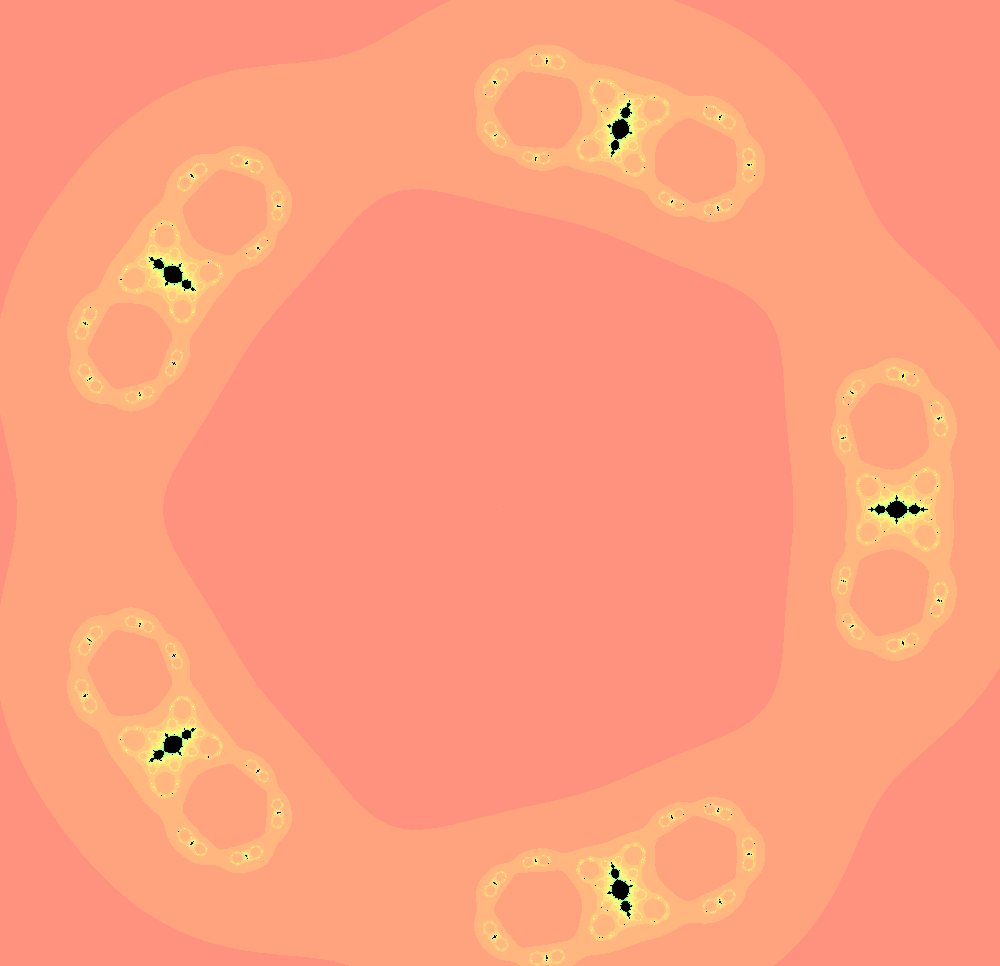} 
\caption{\label{Julia_set_in_Annulus_Figure} The Julia set of $R_{n,a,c}$ with $n=5$, $a \approx 0.7$, and $c = -.75$.}
\end{figure}

Now we define the region on which we will show $R_{n,a,c}$ is polynomial-like of degree two: 
\begin{equation}
\label{UPrime_Equation_Definition}
\boxed{
\setU' = \setU'_{n,a}=\left\{  z=re^{i\theta} \   \middle|  \ \ \frac{\lvert a \rvert^{\frac{1}{n}}}{2}<~r~<2~~and~~\dfrac{\psi-\pi}{2n}<~\theta~<\dfrac{\psi+\pi}{2n} \right\}
}
\end{equation}

where $\psi = \Arg(a)$, and we set
\begin{equation}
\label{U_Equation_Definition}
\boxed{
\setU = \setU_{n,c,a} = R_{n,c,a}(\setU'_{n,a})}~.
\end{equation}

We see that $\setU'$ is slice of $\mathbb{A}\left( \dfrac{\lvert a \rvert^{\frac{1}{n}}}{2},2 \right)$ so it contains a portion the Julia set of $R_{n,a,c}$. $\setU'$ also contains exactly one of the critical points of $R_{n,a,c}$, specifically $\lvert a \rvert^{1 / 2n}e^{\psi / 2n}$, since $\dfrac{\lvert a \rvert^{\frac{1}{n}}}{2}<~\lvert a \rvert^{1 / 2n}~<2$ is true for $\lvert a \rvert < 2^{2n}$.  The range of $a$ we work in is well below that threshold.  The argument of the critical point, $\dfrac{\psi}{2n}$, is the midpoint of the angular range of $\setU'$.  The rest of the critical points of $R_{n,a,c}$ are spread out in intervals of $\dfrac{\pi}{n}$ radians and these don't fall within the angular interval of $\left( \dfrac{\psi - \pi}{2n}~,~\dfrac{\psi + \pi}{2n}\right)$. Thus $\setU'$ contains a unique critical point of $R_{n,a,c}$ and we have established one of the criteria of Definition \ref{Polynomial-like_definition}.

To satisfy the rest of Definition \ref{Polynomial-like_definition}, we start by describing $\setU$ more precisely:

\begin{lemma}
\label{uhalfellipse_lemma}
$\setU$ is half an ellipse centered at $c$ and rotated by $\psi /2$.
\end{lemma}
\begin{proof}
Ignoring the restriction on argument, we consider the set 
$$
\left\{ R_{n,c,a} \left( 2e^{i\theta} \right) \ \middle| \ \ 0\leq \theta \leq 2\pi \right\}.
$$ 
This set contains the image of the outer and inner arcs of $\setU'$ by Lemma \ref{involution_prop}. Because we are considering all angles of $\theta$, our set is independent of the starting angle.  We can apply an angular shift and the image set will remain the same, thus we instead consider the set
$$
\left\{ R_{n,c,a} \left( 2e^{i\left(\theta+\frac{\psi}{2n}\right)} \right) \ \middle| \ \ 0\leq \theta \leq 2\pi \right\}.
$$
So
\begin{eqnarray}
& ~ &R_{n,c,a}\left(2e^{i\left(\theta+\frac{\psi}{2n}\right)}\right) \notag\\
& = & \left(2*\exp \left(i\left(\theta+\frac{\psi}{2n}\right)\right) \right) ^{n}+\frac{a}{\left( 2*\exp \left(i\left(\theta+\frac{\psi}{2n}\right)\right) \right) ^{n}}+c \notag \\
& = & 2^{n}*\exp \left(i\left(n\theta+\frac{\psi}{2}\right)\right)+\frac{\lvert a \rvert e^{i\psi}}{2^{n}*\exp \left(i\left(n\theta+\frac{\psi}{2}\right)\right)}+c\notag \\
& = & 2^{n}*\exp \left(i\left(n\theta+\frac{\psi}{2}\right)\right)+\frac{\lvert a \rvert}{2^{n}}*\exp \left(i\left(\psi - \left(n\theta+\frac{\psi}{2}\right)\right)\right)+c \notag \\
& = & e^{i\frac{\psi}{2}}\left(2^{n}e^{in\theta}+\frac{\lvert a \rvert}{2^{n}}e^{-in\theta}\right)+c \notag \\
 & = & e^{i\frac{\psi}{2}}\left(2^{n}\left(\cos\left( n\theta \right)+i\sin \left(n\theta \right) \right)+\frac{\lvert a \rvert}{2^{n}}\left( \cos \left(n\theta \right)-i\sin \left(n\theta \right)\right)\right)+c \notag \\ 
\label{ellipse} & = & e^{i\frac{\psi}{2}} \left( \left(2^{n}+\frac{\lvert a \rvert}{2^{n}}\right)\cos(n\theta)+i\left(2^{n}-\frac{\lvert a \rvert}{2^{n}}\right)\sin(n\theta) \right) + c.
\end{eqnarray}

Note the above is of the form $e^{i\frac{\psi}{2}} \left( x + iy \right) + c,~\text{where:}$

\begin{equation} 
\label{paraellipse}
\begin{array}{lcl}
& x= & \left(2^{n}+\frac{\lvert a \rvert}{2^{n}}\right)\cos(n\theta) \\
& y= & \left(2^{n}-\frac{\lvert a \rvert}{2^{n}}\right)\sin(n\theta).
\end{array}
\end{equation}
Compare this to the parametric equation of an ellipse centered at the origin:
\begin{eqnarray*}
& x= & b\cos(\phi) \\
& y= & d\sin(\phi)
\end{eqnarray*}
where $0\leq \phi \leq 2\pi$, $b$ is half the length of the major axis and $d$ is half the length of the minor axis.  These axes lie respectively on the real and imaginary axes of the complex plane.

Thus the equation set \eqref{paraellipse} is an ellipse centered at the origin with a major axis length of $2\left( 2^{n}+\frac{\lvert a \rvert}{2^{n}}\right)$, and a minor axis length $2\left( 2^{n}-\frac{\lvert a \rvert}{2^{n}}\right)$ that wraps around $n$ times. Going back to \eqref{ellipse} we find our image set to be the ellipse described above, rotated by $\psi /2$ and centered at $c$.  By our independence of starting angle, this gives us equality to the first set described, and $\left\{ R_{n,c,a}(2e^{i\theta}) \ \middle| \ \ 0\leq \theta \leq 2\pi \right\}$ is this exact ellipse as well. 

 \textbf{Hence we define the ellipse $\ellipse$ by Equation \eqref{ellipse}}.


Now we look at the image of the rays $\displaystyle re^{i\frac{\psi \pm \pi}{2n}}$ for $\frac{\lvert a \rvert^{\frac{1}{n}}}{2}<~r~<2$.

\begin{eqnarray*}
R_{n,c,a} \left(re^{i\frac{\psi \pm \pi}{2n}} \right) & = & \left( r*\exp\left(i\frac{\psi \pm \pi}{2n}\right) \right) ^{n}+\frac{a}{\left( r*\exp\left(i\frac{\psi \pm \pi}{2n}\right) \right) ^{n}}+c  \\
& = & r^{n}*\exp\left(i\frac{\psi \pm \pi}{2}\right)+\frac{\lvert a \rvert e^{i\psi}}{r^{n}*\exp\left(i\frac{\psi \pm \pi}{2}\right)}+c \\
& = & r^{n}*\exp\left(i\frac{\psi \pm \pi}{2}\right)+\frac{\lvert a \rvert}{r^{n}}*\exp\left(i\frac{\psi \mp \pi}{2}\right)+c \\
& = & \exp \left(i\frac{\psi}{2}\right) \left(r^{n}*\exp\left(\pm i\frac{\pi}{2}\right)+\frac{\lvert a \rvert}{r^{n}}*\exp\left(\mp i\frac{\pi}{2}\right)\right)+c \\
& = & \pm e^{i\frac{\psi}{2}}\left(r^{n} - \frac{\lvert a \rvert}{r^{n}}\right)i+c.
\end{eqnarray*}

This is a line segment on the imaginary axis from $-\left(2^{n} - \frac{\lvert a \rvert}{2^{n}}\right)i$ to 
\\ $\left(2^{n} - \frac{\lvert a \rvert}{2^{n}}\right)i$, rotated by $\psi /2$, then shifted by $c$.  In fact, this is actually the minor axis of $\ellipse$.

Finally we investigate the original restriction of $\theta$, 
$$
\dfrac{\psi-\pi}{2n}<~\theta~<\dfrac{\psi+\pi}{2n},
$$
and find that \eqref{paraellipse} gives us 
$$
\dfrac{\psi}{2}-\dfrac{\pi}{2} < n\theta < \dfrac{\psi}{2}+\dfrac{\pi}{2}.
$$ 
We see the angular range is $\pi$ radians in size, hence yielding half an ellipse.  Combine this curve with the minor axis of $\ellipse$ (the image of the rays) and we have that $\setU$ is a half ellipse rotated by $\psi/2$ and centered at $c$. (See Figure \ref{Ellipse_Figure})
\end{proof}

\begin{figure}
\centering
\includegraphics[scale=.5]{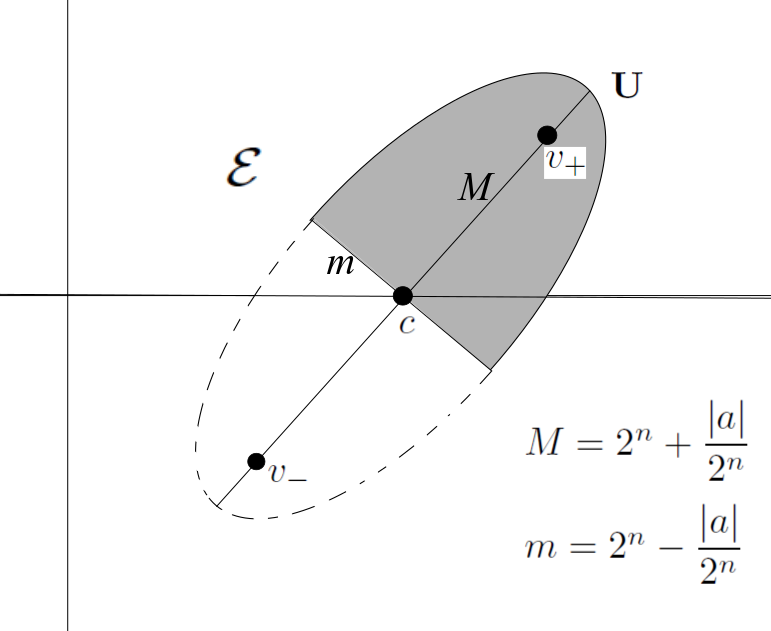} 
\caption{A sketch of $\setU$. It is a half ellipse cut by the minor axis centered at $c$ and rotated by $\frac{\psi}{2}$}
\label{Ellipse_Figure}
\end{figure}

It turns out that the foci of $\ellipse$ are values of importance, in fact, they are the critical values of the map $R_{n,c,a}$.

\begin{lemma}
\label{Ellipse_Foci_Are_CritValues_Lemma}
$v_{\pm}$ are the foci of $\ellipse$.
\end{lemma}

\begin{proof}
For any ellipse, the foci lie on the major axis. The square of the distance of a focal point from the center is equal to difference of the squares of half the major and minor axis lengths.  So we get:
$$
\sqrt{\left(2^n + \dfrac{a}{2^n} \right)^2 - \left(2^n - \dfrac{a}{2^n} \right)^2}= \sqrt{4a} = 2\sqrt{a}.
$$
Since the center of $\ellipse$ is $c$, the foci of $\ellipse$ must be $c \pm 2\sqrt{a} = v_{\pm}$.
\end{proof}

Using Lemma \ref{Ellipse_Foci_Are_CritValues_Lemma} we can describe $\ellipse$ via another equation,
\begin{equation}
\label{Ellipse_Equation_Second_Definition}
\boxed{
\ellipse = \left\{  z \   \middle| ~ \lvert z-v_- \rvert + \lvert z-v_+ \rvert = 2^{n+1} + \frac{\lvert a \rvert}{2^{n-1}} \right\}
}~.
\end{equation}

Being able to describe $\ellipse$ as \eqref{Ellipse_Equation_Second_Definition} helps in the proof of our next lemma.  Now we begin to satisfy more criteria of Definition \ref{Polynomial-like_definition}.

\begin{lemma}
\label{u'inellipse}
Given $n \geq 3$, $\dfrac{c^2}{4}\leq ~\lvert a \rvert~\leq \left( 1-\dfrac{c}{2}\right)^2$, and $\lvert c \rvert \leq 1$ then $\setU' \subseteq \ellipse$.
\end{lemma} 

\begin{proof}
Since $\setU' \subseteq \overline{\mathbb{D}(0,2)}$ we will assume $\lvert z \rvert \leq 2$ and prove $\overline{\mathbb{D}(0,2)} \subseteq \ellipse$, and thus $\setU' \subseteq \ellipse$.

So
\begin{eqnarray*}
& ~ &  \lvert z-v_- \rvert + \lvert z-v_+ \rvert \\
& = & \lvert z-(c-2\sqrt{a}) \rvert + \lvert z-(c+2\sqrt{a}) \rvert \\
& \leq & 2\lvert z \rvert + 2\lvert c \rvert +4\sqrt{\lvert a \rvert} \\
& \leq & 2(2) + 2(1) + 4\sqrt{\left( 1-\frac{c}{2}\right)^2} \\
& \leq &  4 + 2 + 4(1.5) \\
& < & 16 \\
(since ~ n \geq 3)& \leq & 2^{n+1} ~ \leq ~ 2^{n+1} + \frac{\lvert a \rvert}{2^{n-1}}.
\end{eqnarray*}
By Lemma \ref{Ellipse_Foci_Are_CritValues_Lemma} and the description of $\ellipse$ in Equation \eqref{Ellipse_Equation_Second_Definition}, these inequalities yield that $\setU' \subseteq \ellipse$ under the restrictions on $\lvert a \rvert$ and $\lvert c \rvert$.
\end{proof}

Having $\setU'$ contained in $\ellipse$ is helpful, but we need to restrict further to one half of $\ellipse$.  The critical point in $\setU'$ maps to $v_+$ which is in the right half of $\ellipse$. We will restrict the argument of $a$ to $\lvert \psi \rvert \leq \frac{\pi}{n-1}$ which is bounded by $\frac{\pi}{2}$ for $n \geq 3$.  Since $\ellipse$ is a horizontal ellipse rotated by $\frac{\psi}{2}$, then under this restriction $\ellipse$ is rotated by $\frac{\pi}{4}$ at most.  In our next lemma we give criteria for which the minor axis of $\ellipse$ does not intersect $\setU'$ and at worst intersects $\partial \overline{\setU'}$ on the left side.  This will then give us that $\setU' \subset \setU$.

\begin{lemma}
\label{noaxisintercept}
For $n \geq 3$ and real $c < \dfrac{\lvert a \rvert ^{\frac{1}{n}}}{2\sqrt{2}}$, the minor axis of $\ellipse$ does not intersect $\setU'$.
\end{lemma}

\begin{proof}
Keeping the restriction to the argument of $a$, $\lvert \psi \rvert \leq \frac{\pi}{2}$, we determine the value of $c$ for which the minor axis intersects $\setU'$ when $\psi = \pm \frac{\pi}{2}$.

Based on Figure \ref{Ellipse_intersecting_UPrime_Figure}, if we start with $\psi = \frac{\pi}{2}$ then the minor axis will be rotated by $\frac{\pi}{4}$ with respect to the imaginary-axis and then shifted by $c$.  We shall determine the value of $c$ for which the minor axis hits the lower left vertex of $\setU'$ at modulus $\dfrac{\lvert a \rvert^{\frac{1}{n}}}{2}$ and argument $\dfrac{\frac{\pi}{2}-\pi}{2n} ~ = ~ -\dfrac{\pi}{4n}$.
\begin{figure}[htbp]
\centering
\ \ \ \ \ 
\includegraphics[width=1\textwidth,keepaspectratio]{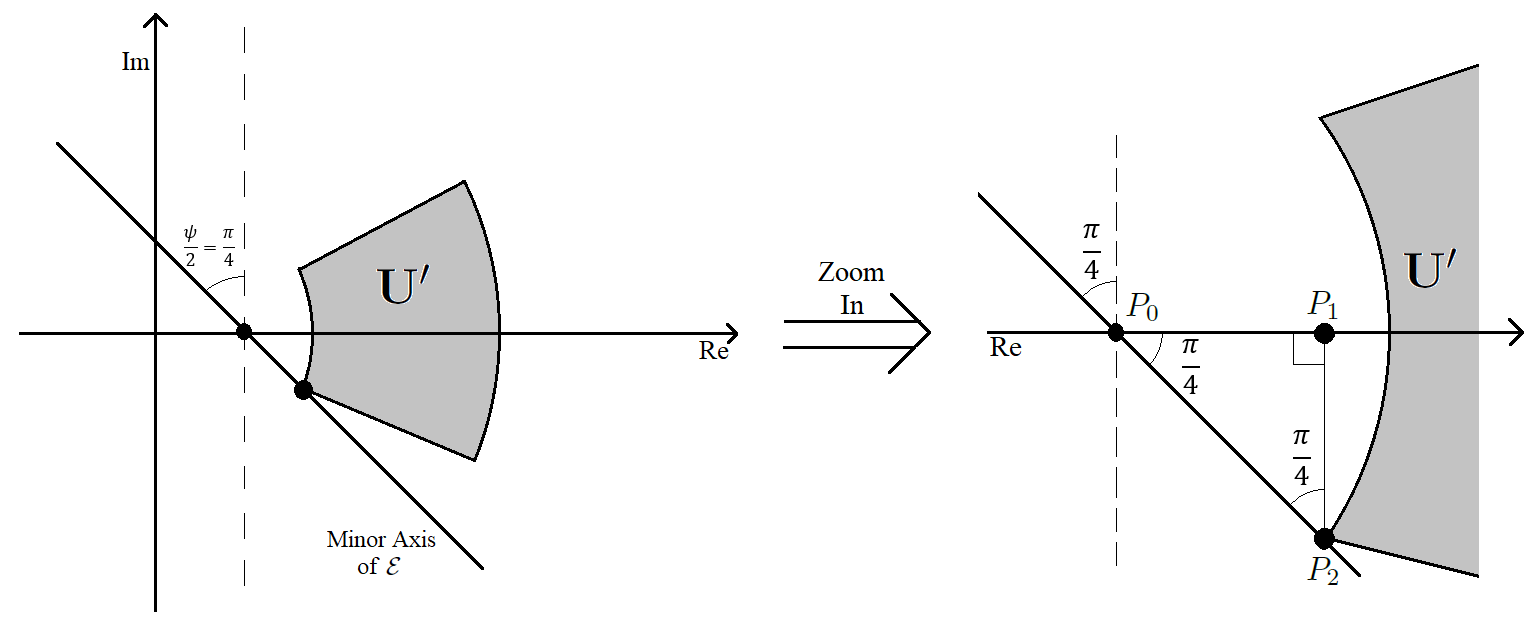} 
\centering
\caption{The minor axis of $\ellipse$ intersecting $\setU'$, with a zoom in shown on the right.}
\label{Ellipse_intersecting_UPrime_Figure}
\end{figure}
The coordinates of this intersection point are\\ 
{$P_2={\left( \dfrac{\lvert a \rvert^{\frac{1}{n}}}{2}\cos\left( \dfrac{\pi}{4n}\right) ~ , ~ -\dfrac{\lvert a \rvert^{\frac{1}{n}}}{2}\sin\left( \dfrac{\pi}{4n}\right) \right)}$}.  
With a closer view, we have a right triangle as shown on the right of Figure \ref{Ellipse_intersecting_UPrime_Figure} and because $\frac{\psi}{2} = \frac{\pi}{4}$, both legs of the triangle are equal length.  The coordinates of the other two points of the triangle are

\begin{itemize}
\item $P_0=\left(0 ,c \right)$,
\item $P_1=\left( \dfrac{\lvert a \rvert^{\frac{1}{n}}}{2}\cos\left( \dfrac{\pi}{4n}\right) ~ , ~0 \right)$.
\end{itemize}

The length of the vertical leg is the absolute value of the imaginary component of $P_2$. The length of the horizontal leg is the difference between the real components of $P_0$ and $P_1$.  We set these values equal to each other,
$$
\dfrac{\lvert a \rvert^{\frac{1}{n}}}{2}\sin\left( \dfrac{\pi}{4n}\right) ~ = ~ \dfrac{\lvert a \rvert^{\frac{1}{n}}}{2}\cos\left( \dfrac{\pi}{4n}\right) - c,\\
$$
then solve for $c$ in terms of $a$:
\begin{eqnarray*}
c & = & \dfrac{\lvert a \rvert^{\frac{1}{n}}}{2}\left( \cos\left( \dfrac{\pi}{4n}\right) - \sin\left( \dfrac{\pi}{4n}\right) \right)\\
& \geq & \dfrac{\lvert a \rvert^{\frac{1}{n}}}{2}\left( \cos\left( \dfrac{\pi}{12}\right) - \sin\left( \dfrac{\pi}{12}\right) \right) ~(\text{Since}~n~\geq~3)\\
& = & \dfrac{\lvert a \rvert ^{\frac{1}{n}}}{2\sqrt{2}}.
\end{eqnarray*}

Thus the smallest $c$ value for which the minor axis of $\ellipse$ intersects $\setU'$ is $c = \dfrac{\lvert a \rvert ^{\frac{1}{n}}}{2\sqrt{2}}$, so choosing $-1 \leq c < \dfrac{\lvert a \rvert ^{\frac{1}{n}}}{2\sqrt{2}}$ yields $\partial \setU' \cap \partial \setU = \emptyset$.

If we go to the other extreme and let $\psi = -\frac{\pi}{2}$, the work will be just the same. Here our intersection point is now  $\left( \dfrac{\lvert a \rvert^{\frac{1}{n}}}{2}\cos\left( \dfrac{\pi}{4n}\right) ~ , ~ \dfrac{\lvert a \rvert^{\frac{1}{n}}}{2}\sin\left( \dfrac{\pi}{4n}\right) \right)$, and the triangle is just a reflection of the previous case across the real axis.  Therefore our result is the same as above.

Finally, we argue that for $c=\dfrac{\lvert a \rvert ^{\frac{1}{n}}}{2\sqrt{2}}$, the minor axis will only touch the corners of $\setU'$ at $\psi = \pm \frac{\pi}{2}$ and there will be no intersection for any $\psi$ or $c$ value smaller than these bounds.

\begin{figure}
\centering
\includegraphics[width=\linewidth,keepaspectratio]{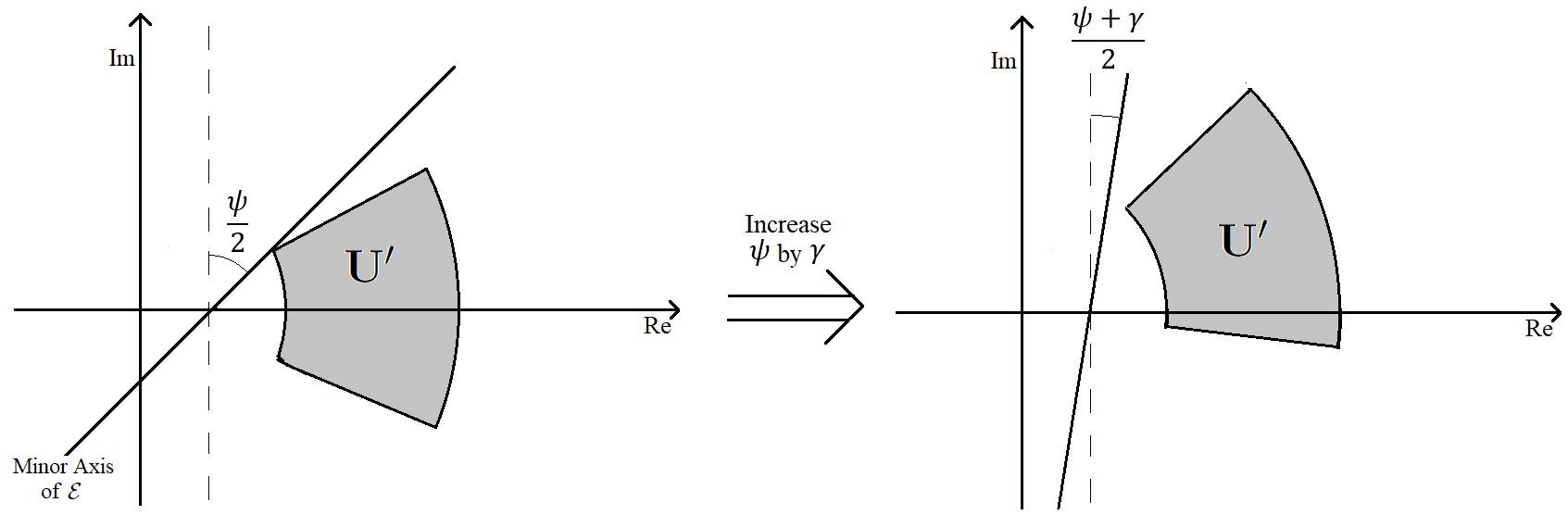} 
\caption{\label{fig:axisnotintersect} Increasing the angle results in no intersection between $\setU'$ and the minor axis of the ellipse $\ellipse$ .}
\end{figure}

When we change $\psi$, the overall change in angle of the minor axis will be larger than the overall change in angle of the rays of $\setU'$. Starting at $\psi = -\frac{\pi}{2}$ and increasing by some value $\gamma$, we find the change in angle of the minor axis to be

$$
\dfrac{\psi + \gamma}{2}-\dfrac{\psi}{2}=\dfrac{\gamma}{2}
$$

and the change in angle of the upper ray of $\setU'$ to be 
$$
\dfrac{(\psi + \gamma)+\pi}{2n}-\dfrac{\psi + \pi}{2n}=\dfrac{\gamma}{2n}.
$$

Since $n \geq 3$, then $\frac{\gamma}{2} > \frac{\gamma}{2n}$ and the minor axis won't touch $\setU'$ again until it goes too far and hits the ``lower left" corner of $\setU'$.  As shown above though, this won't occur until $\psi = \frac{\pi}{2}$.\\

Therefore for $c ~ < ~ \dfrac{\lvert a \rvert ^{\frac{1}{n}}}{2\sqrt{2}}$ we find the minor axis of $\ellipse$ will not intersect $\setU'$.
\end{proof}

Figure \ref{UPrime_Subset_U_Graphic} shows one example of $\setU' \subset \setU$ (there is nothing particularly special about these parameter values, they are merely round numbers which satisfy $\setU' \subset \setU$ ).  

\begin{figure}
\centering
\includegraphics[width=.5\textwidth,keepaspectratio]{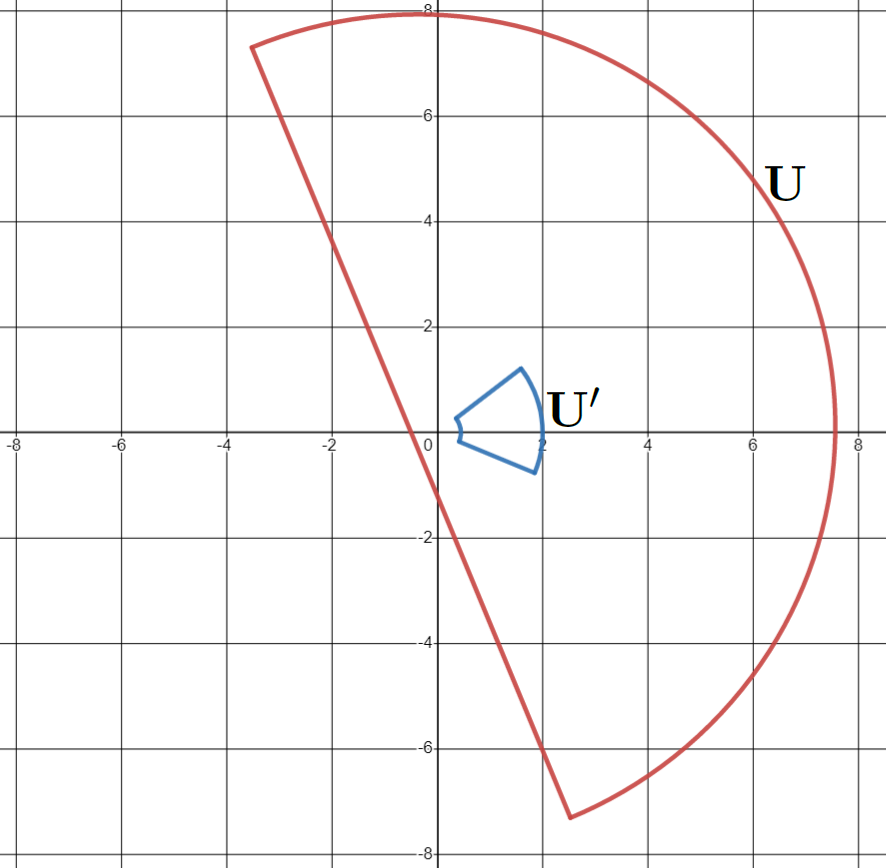} 
\caption{An example of $\setU'$ sitting inside $\setU$. Here $n=3$, $a=0.5 + 0.5i$, and $c=-0.5$}
\label{UPrime_Subset_U_Graphic}
\end{figure}

Now we can show that $R_{n,a,c}$ is polynomial-like on $\setU'$.

\begin{proposition}
\label{R_Polynomial_like_on_first_U'_aplane_prop}
$R_{n,a,c}: \setU' \rightarrow \setU$ is polynomial-like of degree two when $n \geq 3$,
$\dfrac{c^2}{4}\leq ~\lvert a \rvert~\leq \left( 1-\dfrac{c}{2}\right)^2$, and $-1 \leq c \leq 0$.
\end{proposition}

\begin{proof}
Both $\setU'$ and $\setU$ are bounded, open, and simply connected. Combining the results from Lemma \ref{u'inellipse} and Lemma \ref{noaxisintercept} with the restriction that $-1 \leq c \leq 0 < \dfrac{\lvert a \rvert ^{\frac{1}{n}}}{2\sqrt{2}}$ yields that $\setU'$ is relatively compact in $\setU$.  Also, Lemma \ref{involution_prop} plus the fact that $\setU'$ contains a unique critical point yields that $R_{n,a,c}$ is a two-to-one map on $\setU'$.
Last, $R_{n,a,c}$ is analytic on $\setU'$ because $a \neq 0$ and $\setU'$ does not contain the origin. Therefore $R_{n,a,c}$ satisfies Definition \ref{Polynomial-like_definition} and is polynomial-like of degree two.  
\end{proof}

Now because of the polynomial-like behavior of $R_{n,a,c}$, we see the reason for baby quadratic Julia sets appearing in the dynamical plane of $R_{n,a,c}$.

\begin{corollary}
\label{Baby_Julia_Corollary_First}
With the same assumptions on $n$, $a$, and $c$ as in Proposition \ref{R_Polynomial_like_on_first_U'_aplane_prop}, the collection of points in $\setU'$ whose orbits do not escape $\setU'$ is homeomorphic to the filled Julia set of a quadratic polynomial.
\end{corollary}

\begin{proof}
Having established that $R_{n,a,c}$ is polynomial-like of degree two in Proposition \ref{R_Polynomial_like_on_first_U'_aplane_prop}, by Theorem \ref{DH_Ploynomial-Like_theorem}{} $R_{n,a,c}$ is topologically conjugate to some quadratic polynomial on that polynomial's filled Julia set.  This polynomial's Julia set is the baby Julia set we find in the dynamical plane of $R_{n,a,c}$.
\end{proof}

This confirms that the five obvious black shapes in Figure \ref{Julia_set_in_Annulus_Figure} are baby Julia sets occurring in the dynamical plane of $R_{n,a,c}$.

\subsection{Parameter Plane Results: $a$-plane}

Now we turn to locating a homeomorphic copy of $\mandel$ in the boundedness locus in the $a$-parameter plane of $R_{n,a,c}$. We need to show we satisfy the criteria of Theorem \ref{DH_Mandelbrot_existence_Criterion_Theorem}, so we first define the set $\mathbf{W}$ mentioned in its hypothesis:


\begin{equation}
\boxed{
\setWa = \left\{  a \   \middle|  \ \ \frac{c^2}{4}\leq ~\lvert a \rvert~\leq \left( 1-\frac{c}{2}\right)^2~~and~~\lvert \psi \rvert \leq \frac{\pi}{n-1} \right\}}~.
\label{eqn:defnW}
\end{equation}

Remember that $\psi=\Arg(a)$. Now we show how the other parts of the hypotheses from Theorem~\ref{DH_Mandelbrot_existence_Criterion_Theorem} are satisfied.

\begin{proposition}
\label{allnvplusloopW}
For $n \geq 3$ and $-1 \leq c \leq 0$, as $a$ travels around $\partial \setWa$, $v_+$ loops around $\setU - \setU'$.
\end{proposition}

\begin{proof}
We start with $\lvert a \rvert = \dfrac{c^2}{4}$ on the inside arc of $\setWa$. Here
$$
Re(v_+)=Re\left( c+2\sqrt{a} \right) = c+2\sqrt{\lvert a \rvert}\cos\left( \dfrac{\psi}{2} \right) = c+\lvert c \rvert \cos \left( \dfrac{\psi}{2} \right) \leq c+ \lvert c \rvert = 0
$$ 
since $c \leq 0$.  This means $v_+$ never lies in the right half plane. The inside arc of $\setU'$ is always in the right half plane because 
$$
\lvert \Arg(z \in \setU') \rvert \leq \left| \dfrac{\pm \frac{\pi}{n-1}\pm \pi}{2n} \right| \leq \dfrac{\pi}{2(n-1)} < \dfrac{\pi}{2}
$$ 
for $\lvert \Arg(a) \rvert \leq \frac{\pi}{n-1}$ and $n \geq 3$. This means that $v_+$ lies to the left of $\setU'$ when $a$ is on the inside arc of $\setWa$.

Now for $a$ on the upper ray of $\setWa$,~ $\Arg(a)=\frac{\pi}{n-1}$. Thus for any $z \in \setU'$, 

$$
\Arg(z) \leq \frac{\pi}{2(n-1)}.
$$  

Because $c$ is real and non-positive, we find
\begin{eqnarray*}
 \Arg(v_+) & = & \Arg(c+2\sqrt{a})\\
 & \geq & \Arg(2\sqrt{a}) \\
 & = & \dfrac{\psi}{2} \\
 & = & \dfrac{1}{2}\dfrac{\pi}{n-1} \\
 & \geq & \Arg(z \in \setU').
\end{eqnarray*}
So at worst $v_+$ touches the upper ray of $\setU'$ when $a$ is on the upper ray of $\setWa$.  This is permissible as $\setU'$ is open.

When we reach the outer arc of $\setWa$ where $\lvert a \rvert = \left( 1-\frac{c}{2}\right)^2$, we find that 
$\lvert v_+ \rvert = \lvert c+2\sqrt{a} \rvert $
$$
= \sqrt{c^2 + 4\lvert a \rvert + 4c\sqrt{\lvert a \rvert}\cos \left(\dfrac{\psi}{2} \right)} = \sqrt{4-4c+2c^2-2c(c-2)\cos \left(\dfrac{\psi}{2} \right)}.
$$

To find the minimum of this, we take a derivative with respect to $\psi$ and get 
$$
\dfrac{c(c-2)\sin\left( \frac{\psi}{2} \right)}{2\sqrt{4-4c+2c^2-2c(c-2)\cos \left( \frac{\psi}{2} \right)}},
$$ 
and the modulus has a critical point at $\psi=0$. Using the second derivative test we find a minimum occurs at $\psi=0$, and evaluating the modulus at $\psi = 0$ gives 
$$
\lvert v_+ \rvert = \sqrt{4-4c+2c^2-2c(c-2)(1)}=2.
$$

Therefore, at worst $v_+$ just touches the outer boundary arc of $\setU'$ as $a$ travels along the outer arc of $\setWa$ (still permissible with $\setU'$ open).

Now for $a$ on the lower ray, $\Arg(a)=-\frac{\pi}{n-1}$. Using this and the fact that $\Arg(c)=-\pi$, we find 
$$
\Arg(z \in \setU') > -\dfrac{1}{2} \dfrac{\pi}{n-1} = \Arg(2\sqrt{a}) \geq \Arg(c+2\sqrt{a}).
$$ 

So $\Arg(z \in \setU') > \Arg(v_+)$ for all $a$ on the lower ray of $\setWa$ and $v_+$ lies below $\setU'$. 

Therefore as $a$ comes back to its starting position in $\partial \setWa$, $v_+$ has finished a closed loop around the outside of $\setU'$.
\end{proof}

With Propositions \ref{R_Polynomial_like_on_first_U'_aplane_prop} and \ref{allnvplusloopW} we satisfy the necessary conditions of Theorem \ref{DH_Mandelbrot_existence_Criterion_Theorem} for the existence of a baby $\mandel$ lying in $\setWa$, in a subset of the boundedness locus.

\begin{theorem}
\label{V+Mandel_APlane_Cnegative_theorem}
For $n \geq 3$ and $-1 \leq c \leq 0$, the set of $a$-values within $\setWa$ for which the orbit of $v_+$ does not escape $\setU'$ is homeomorphic to $\mandel$.
\end{theorem}

See Figure \ref{baby_mandel_in_W_Figure} for an example of such a baby $\mandel$. 

Theorem \ref{V+Mandel_APlane_Cnegative_theorem} is part of Main Theorem \ref{Main_Theorem_APlane}. We will finish Main Theorem \ref{Main_Theorem_APlane} by taking advantage of some symmetries in the family $R_{n,a,c}$.

\begin{figure}
\centering
\includegraphics[width=0.75\textwidth,keepaspectratio]{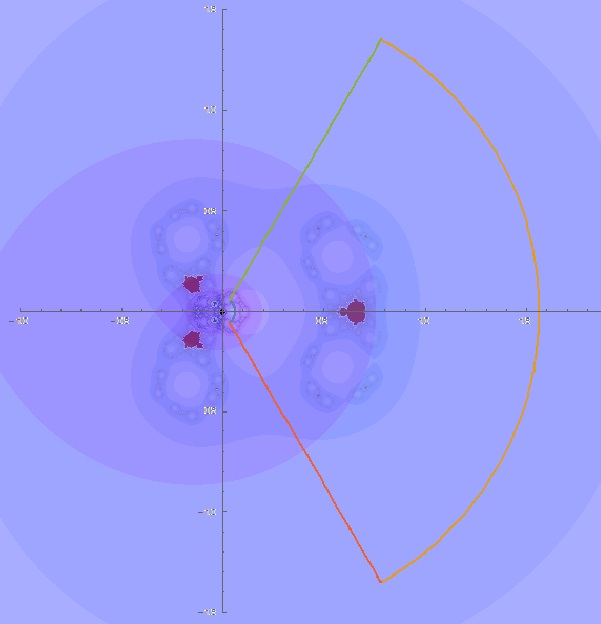}
\caption{A baby $\mandel$ sitting in $\setWa$ in the $a$-parameter plane of $R_{n,a,c}$ for $n=4$ and $c=-0.5$.}
\label{baby_mandel_in_W_Figure} 
\end{figure}

\begin{lemma}
\label{symmetry_through_negz_negc}
For $n$ odd, $R^{m}_{n,a,-c}(-z) = -R^{m}_{n,a,c}(z)$ for all $m \in \mathbb{N}$.
\end{lemma}

\begin{proof}
We prove this inductively, so let the base case be the first iterate of $R_{n,a,c}$: 

\begin{eqnarray*}
R_{n,a,-c}(-z)&=& (-z)^n + \dfrac{a}{(-z)^n} - c\\
&=& -z^n - \dfrac{a}{z^n} - c \hspace{1cm}  \text{(Since n is odd)}\\
&=& -\left( z^n + \dfrac{a}{z^n} + c \right)\\
&=& -R_{n,a,c}(z).\\
\end{eqnarray*}

Now assuming the hypothesis is true for $m-1$, $R^{m-1}_{n,a,-c}(-z) = -R^{m-1}_{n,a,c}(z)$, then

\begin{eqnarray*}
R^{m}_{n,a,-c}(-z)&=& R^{m-1}_{n,a,-c}(R_{n,a,-c}(-z))\\
&=& R^{m-1}_{n,a,-c}(-(R_{n,a,c}(z))) \hspace{1cm} \text{(from the base case)}\\
&=& -R^{m-1}_{n,a,c}(R_{n,a,c}(z))\hspace{1cm} \text{(from the inductive assumption)}\\
&=& -R^{m}_{n,a,c}(z).\\
\end{eqnarray*}

And this result holds for all iterates of $R_{n,a,c}$.
\end{proof}

\begin{lemma}
\label{Crit_Orbits_Symm_Through_NegC_Lemma_APlane}
For $n$ odd, $R^{m}_{n,a,-c}(v_-) = -R^{m}_{n,a,c}(v_+)$, that is that the behavior of the critical orbits are symmetric through $c$ and $-c$.
\end{lemma}

\begin{proof}
Using Lemma \ref{symmetry_through_negz_negc} we find for every positive integer $m$,
\begin{eqnarray*}
R_{n,a,-c}^m(v_-) &=& R_{n,a,-c}^m(-c-2\sqrt{a})\\ 
&=& R_{n,a,-c}^m(-(c+2\sqrt{a})) = -R_{n,a,c}^m(c+2\sqrt{a}) = -R_{n,a,c}^m(v_+)
\end{eqnarray*}

and the critical orbits are symmetric about $c$ and $-c$.
\end{proof}

Lemma \ref{Crit_Orbits_Symm_Through_NegC_Lemma_APlane}  says that the boundedness locus in the $a$-parameter plane for $c$ and $-c$ are the same when $n$ is odd. With this we can combine our results to gain existence of another baby $\mandel$.


\begin{corollary}
\label{APlane_V-_Corollary}
For odd $n \geq 3$ and $0 \leq c \leq 1$, the set of $a$-values within $\setWa$ for which the orbit of $v_-$ does not escape $\setU'$ is homeomorphic to $\mandel$.
\end{corollary}

\begin{proof}
The proof of this comes from the existence of the $v_+$ baby $\mandel$ in Theorem \ref{V+Mandel_APlane_Cnegative_theorem} and the symmetry of the critical orbits in Lemma \ref{Crit_Orbits_Symm_Through_NegC_Lemma_APlane}.
\end{proof}

Figure \ref{aPlane_n5_vMinus_Mandel_Figure} shows an example $a$-plane of $R_{n,a,c}$ showing the existence of this $v_-$ baby $\mandel$ in the same spot that a $v_+$ baby $\mandel$ would be promised by the symmetry. This result finishes Main Theorem \ref{Main_Theorem_APlane}.

\begin{figure}
\centering
\includegraphics[width=0.75\textwidth,keepaspectratio]{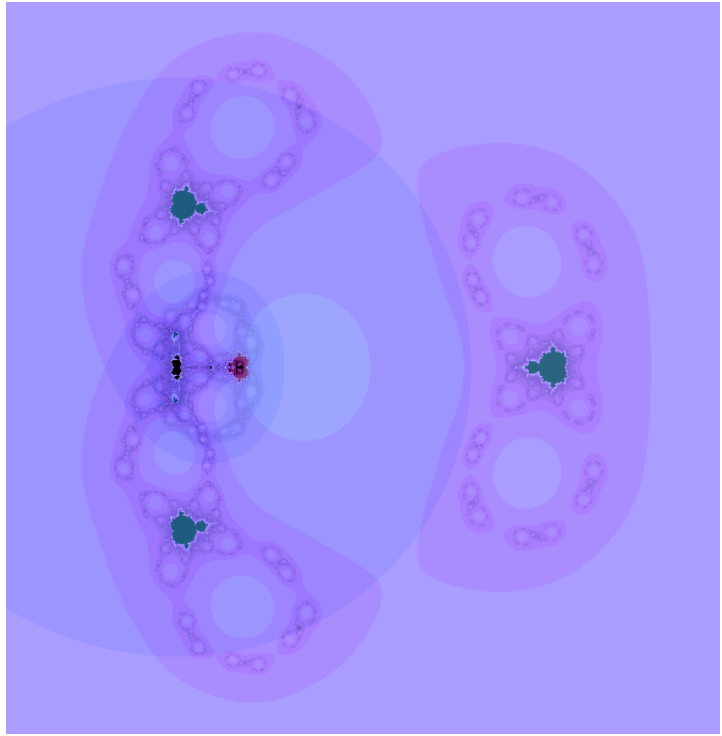}
\caption{\label{aPlane_n5_vMinus_Mandel_Figure}The $a$-parameter plane of $R_{n,a,c}$ with $n=5$ and $c=0.5$. The baby $\mandel$ associated with $v_-$ is on the right side of the picture.}
\end{figure}

\section{The case of $a$ fixed, $c$ varying}
\label{Main_Two_Section}

In this section, our goal is to establish Main Theorem \ref{Main_Theorem_CPlane}, item (i).
We work under the following assumptions:
\begin{itemize}
\item $n \geq 5$,
\item $1 \leq a \leq 4$,
\item $c$ chosen such that $\lvert v_+ \rvert \leq 2$.
\end{itemize}

\subsection{Dynamical Plane Results} 

As with the case of $c$ fixed and $a$ varying, we start with some results about the dynamical plane under these new restrictions, and use the result from \cite{boydschul} to restrict where the Julia set of $R_{n,a,c}$ may lie under these new conditions.

\begin{lemma}
\label{EscapeRadius_lemma_cplaneCase}
If $n \geq 5$, $1 \leq a \leq 4$, and $c$ is chosen such that $\lvert v_+ \rvert \leq 2$, then the filled Julia set of $R_{n,a,c}$ lies in $\mathbb{D}(0,2)$.
\end{lemma}

\begin{proof}
Once again, by the proof of Lemma \ref{Julia_Set_Restriction_lemma} in \cite{boydschul} for any $\epsilon > 0$, if $N$ satisfies $(1+\epsilon)^N > 3\text{Max} \lbrace 1,\lvert a \rvert, \lvert c \rvert \rbrace$ then for $n \geq N$ we have the escape radius of $1+\epsilon$. That is, the orbits of values $\lvert z \rvert > 1+\epsilon$ escape to $\infty$.  We set $\epsilon=1$.  
The largest possible $\lvert c \rvert$  such that $\lvert v_+ \rvert<2$ is $c = -6$, since $\lvert  -6+2\sqrt{4} \rvert = 2$.  Therefore the modulus of $c$ is bounded by $6$. hence, 
$$
3\text{Max} \lbrace 1,\lvert a \rvert, \lvert c \rvert \rbrace = 3\text{Max} \lbrace 1, 4, 6 \rbrace = 18.
$$
So when we solve $(1+\epsilon)^N = 2^N > 18$ for $N$ we find $N > \dfrac{\ln(18)}{\ln(2)} \approx 4.17$ and $n \geq 5$ will satisfy the criterion.  Because the filled Julia set of $R_{n,a,c}$ is the points whose orbits are bounded, then the Julia set must lie within $\mathbb{D}(0,2)$.
\end{proof}

Combining this with Lemma \ref{involution_prop} again yields the same annulus as in the $a$-plane case.

\begin{lemma}
\label{EscapeAnnulus_lemma_cplaneCase}
With the same assumptions on $n$, $a$, and $c$ as Lemma \ref{EscapeRadius_lemma_cplaneCase}, the filled Julia set of $R_{n,a,c}$ lies within the annulus $\mathbb{A}\left(\dfrac{\lvert a \rvert^{\frac{1}{n}}}{2} ,~2\right)$.
\end{lemma}

In this case of $a$ fixed, when observing the $c$-parameter plane we can actually show there exist $2n$ baby $\mandel$'s. Each baby $\mandel$ has a unique corresponding $\setU'$ that is a rotational copy of our original definition from Equation \eqref{UPrime_Equation_Definition}:


\begin{equation}
\boxed{
\setU'_k = \setU'_{n,k,a}=\left\{  z=re^{i\theta} \   \middle|  \ \ \frac{\lvert a \rvert^{\frac{1}{n}}}{2}<~r~<2, \ \ \dfrac{4\pi k-\pi}{2n}<~\theta~<\dfrac{4\pi k+\pi}{2n} \right\},}
\end{equation}

for $k=0,...,n-1$. Note that $\setU'_0$ is the same as Equation \eqref{UPrime_Equation_Definition} in the case the $a$ is real, hence $\psi = 0$.  Each $k$ represents a different rotational copy inside $\mathbb{A}\left(\dfrac{\lvert a \rvert^{\frac{1}{n}}}{2} ,~2\right)$.  Now we define $n$ $\mathbf{W}$ sets in the $c$-parameter plane that will contain baby $\mandel$s (we will produce $n$ more using symmetry):


\begin{equation}
\boxed{
\setWck = 
\left\{  c \   \middle|  \ \ \dfrac{a^{1/n}}{2} \leq \left| v_+ \right| \leq 2,\ \  \dfrac{4\pi k-\pi}{2n} \leq \Arg(v_+) \leq \dfrac{4\pi k+\pi}{2n} \right\}}.
\label{eqn:defnWck}
\end{equation}

for $k=0,1,..,n-1$. For a fixed $k$, $\setWck$ is the set of $c$ values such that $v_+ \in \overline{\setU'_k}$. Thus each $\setWck$ is associated with a unique $\setU'_k$.  Conveniently, each $\setU'_k$ maps to the same image.

\begin{lemma}
$R_{n,a,c}(\setU'_k)=R_{n,a,c}(\setU'_0)$ for all $k$ in $\lbrace 0,1,...,n-1 \rbrace$.
\end{lemma}

\begin{proof}
The image of the outside and inside curves of each $\setU'_k$ will still map to the same curve in $\setU$ as before by Lemma \ref{involution_prop}, so we examine the image of the rays of $\setU'_k$.

\begin{eqnarray*}
&~& R_{n,c,a}\left( r*\exp \left(i\frac{4\pi k \pm \pi}{2n} \right) \right)\\
& = & \left( r*\exp \left(i\frac{4\pi k \pm \pi}{2n} \right) \right) ^{n}+\frac{a}{\left( r*\exp \left(i\frac{4\pi k \pm \pi}{2n} \right) \right) ^{n}}+c\\
& = & r^{n}*\exp \left(i\frac{4\pi k \pm \pi}{2} \right)+\frac{a}{r^{n}*\exp \left(i\frac{4\pi k \pm \pi}{2} \right)}+c\\
& = & r^{n}*\exp \left(i\frac{4\pi k \pm \pi}{2} \right)+\frac{a}{r^{n}}*\exp \left(i\frac{4\pi k \mp \pi}{2} \right)+c\\
& = & \exp \left(i{2\pi k} \right) \left(r^{n}*\exp \left(\pm i\frac{\pi}{2} \right)+\frac{a}{r^{n}}*\exp \left(\mp i\frac{\pi}{2} \right) \right)+c \\
& = & \pm \left(r^{n} - \frac{a}{r^{n}}\right)i+c.
\end{eqnarray*}

This being equal to the image of the rays of $\setU'_0$ yields our result and
$R_{n,a,c}(\setU'_k)=R_{n,a,c}(\setU'_0)$ for all $k$ in $\lbrace 0,1,...,n-1 \rbrace$.
\end{proof}

Now we continue the process by showing that $R_{n,a,c}$ is polynomial-like of degree two on each $\setU'_k$.  First we show that each $\setU'_k$ is contained in $\ellipse$.

\begin{lemma}
\label{Uprime_contained_in_ellipse_afixed_case_lemma}
Let $n \geq 5$, $1 \leq a \leq 4$ and $c \in \setWck$ for $k=0,1...n-1$. Then $\overline{\setU'_k} \subset \ellipse$.   
\end{lemma}

\begin{proof}
Since each $\setU'_k \subset \mathbb{D}(0,2)$, we just need to prove $\overline {\mathbb{D}(0,2)} \subset \ellipse$ under these restrictions of parameters. Since $c$ lies in some $\setWck$, we have $v_+ \in \setU'_k$.  This means $v_+$ has a maximum modulus of 2 so we start with the equation
$$
c+2\sqrt{a}=2e^{i\theta},~\theta \in [0,2\pi).
$$
Solving for $c$ gets us $c=2e^{i\theta}-2\sqrt{a}$ and we attain bounds for $v_-$ as well:
\begin{eqnarray*}
& ~ & \lvert v_+ \rvert = 2 \\
& \Rightarrow & \lvert v_- \rvert = \lvert c-2\sqrt{a} \rvert = \lvert ( 2e^{i\theta}- 2\sqrt{a} ) - 2\sqrt{a} \rvert = \sqrt{4+16a-16 \cos(\theta) \sqrt{a}}.\\
\end{eqnarray*}
To find the largest value of $\lvert v_- \rvert$, we take a derivative with respect to $\theta$: 
$$
\dfrac{8\sqrt{a}\sin(\theta)}{\sqrt{4+16a-16 \cos(\theta) \sqrt{a}}}.
$$ 
This means $\lvert v_- \rvert$ has critical points at $\theta=\lbrace 0,\pi\rbrace \in [0,2\pi)$. By the second derivative test, $\theta=\pi$ gives us a maximum modulus of $2+4\sqrt{a}$.

As in the proof of Lemma \ref{u'inellipse}, we use the Equation \eqref{Ellipse_Equation_Second_Definition} description of $\ellipse$. Using the facts that $\lvert z \rvert \leq 2$, $\lvert c+2\sqrt{a} \rvert \leq 2$ and $\lvert c-2\sqrt{a} \rvert \leq 2 + 4\sqrt{a}$ we find
\begin{eqnarray*}
&~& \lvert z-v_- \rvert + \lvert z-v_+ \rvert\\
&=& \lvert z-(c-2\sqrt{a}) \rvert + \lvert z-(c+2\sqrt{a}) \rvert\\
& \leq & 2\lvert z \rvert + \lvert c+2\sqrt{a} \rvert + \lvert c-2\sqrt{a} \rvert \\
& \leq & 4 + (2) + (2 + 4\sqrt{a})\\
& = & 8 + 4\sqrt{4} \\
& = &  16 \\
(\text{thus since}~n \geq 5)& < & 2^{n+1} \leq ~ 2^{n+1} + \frac{\lvert a \rvert}{2^{n-1}}.
\end{eqnarray*}
Hence $\setU'_k \subset \overline{\mathbb{D}(0,2)} \subset \ellipse$ for each $k=0,1,...n-1$.
\end{proof}

Continuing further, we prove that each $\setU'_k$ is contained in its image $\setU$ under the correct restrictions of parameters.

\begin{lemma} 
\label{UPrimeK_Contained_ForAllK_Lemma}
Let $n \geq 5$, $1 \leq a \leq \dfrac{(2^{n+1}-8)^2}{16}$, and fix $k$ from $\lbrace 0,1,...,n-1 \rbrace$. If~$c \in \setWck$, then $\setU'_k \subset \setU$.
\end{lemma}

\begin{proof}

By the proof of Lemma \ref{Uprime_contained_in_ellipse_afixed_case_lemma}, $\setU'_k \subset \mathbb{D}(0,2) \subset \ellipse$ for all $k$, so we need to show the minor axis of $\ellipse$ just intersects $\partial \setU'_k$ at worst.

\textbf{CASE 1: $\setU'_k$ lies in the right half plane:}

Here the minor axis of $\ellipse$ is a vertical line since $\frac{\Arg(a)}{2}=0$ and crosses through $c$.  Thus the value of $Re(c)$ determines the horizontal position of the leftmost point of $\setU$. In the proof of Lemma~\ref{Uprime_contained_in_ellipse_afixed_case_lemma} we found for $\lvert v_+ \rvert = 2$ (its greatest modulus) that $c=2e^{i\theta}-2\sqrt{a}$.  Therefore 
$$
Re(c)=2\cos(\theta) - 2\sqrt{a} \leq 2 - 2 = 0
$$
since $a \geq 1$.  Thus if $Re(c)$ is non-positive, then the minor axis of $\ellipse$ is never in the right half dynamical plane.  Therefore the minor axis never enters the right-half plane and $\setU'_k \subset \setU$ for any $\setU'_k$ in this case. (See Figure \ref{fig:uprime_k_cases} (left).)

\begin{figure}[htbp]
\centering
\includegraphics[width=.45\textwidth]{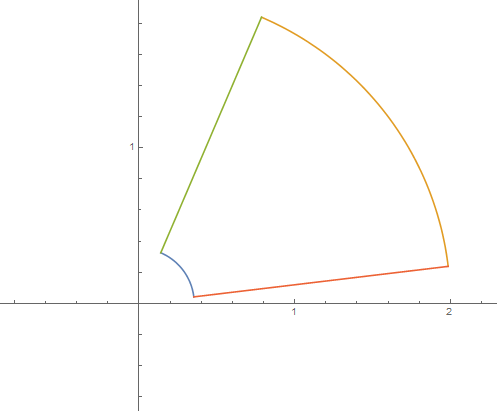} 
\includegraphics[width=.45\textwidth]{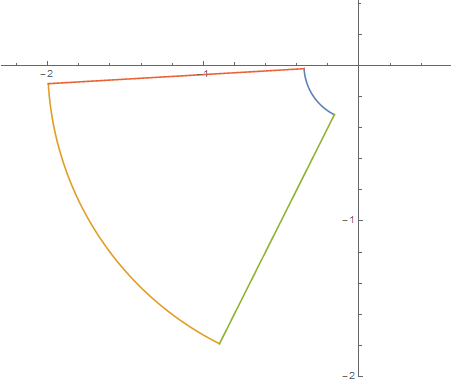} 
\caption{Cases 1 (left) and 2 (right) of Lemma 15.}
\label{fig:uprime_k_cases}
\end{figure}

\textbf{CASE 2:   $\setU'_k$ lies in the left half plane:}

Now $\setU'_k$ lies between the real values of $-2$ and $0$.  Thus at its greatest,
$$
Re(v_+)=0 \Rightarrow Re(c)=-2\sqrt{a} \leq -2
$$ 
since $a \geq 1$. Here the position of the minor axis will go no farther right than $R(c) = -2$ and the minor axis only touches $\partial \setU'_k$ which is admissible since $\setU'_k$ is open.  Thus $\setU'_k \subset \setU$ for any $\setU'_k$ lying purely in the left-half plane. (See Figure \ref{fig:uprime_k_cases} (right).)

\textbf{CASE 3:  $\setU'_k$ intersects the imaginary axis:}

The angular width of $\setU'_k$ is 
$$
\dfrac{4\pi k+\pi}{2n} - \dfrac{4\pi k-\pi}{2n}=\dfrac{\pi}{n} \leq \dfrac{\pi}{3}
$$
since $n \geq 3$.  Assuming for now $\setU'_k$ intersects the positive imaginary axis, we define a set:
$$
\setS = \left\{  z \   \middle|  \ \ \lvert z \rvert \leq 2~~and~~\dfrac{\pi}{6}+t\dfrac{\pi}{3} \leq \Arg(z)\leq \dfrac{\pi}{2}+t\dfrac{\pi}{3} \right\} ~ for ~ t \in [0,1].
$$

For any $t$, $\setS$ is a wedge of angular width $\frac{\pi}{3}$ that intersects the imaginary axis so any $\setU'_k$ that intersects the imaginary axis must be contained in a $\setS$ for some $t \in [0,1]$. Figure \ref{fig:uprime_k_case3} is a sketch of $\setS$, the dotted line represents the range of $\setS$.

\begin{figure}
\centering
\includegraphics[scale=.3]{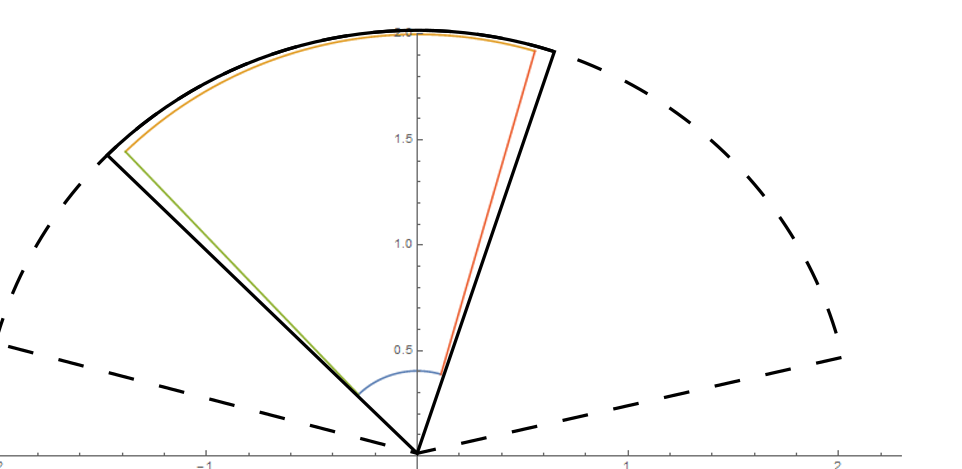} 
\caption{\label{fig:uprime_k_case3} $\setS$ containing a $\setU'_{k}$ that intersects the imaginary axis.}
\end{figure}

In determining the \textit{real-diameter} of $\setS$, we calculate the distance between the left-most and right-most points of $\setS$.  The leftmost point lies on the ray of argument $\frac{\pi}{2}+t\frac{\pi}{3}$ while the rightmost point lies on the ray of argument $\frac{\pi}{6}+t\frac{\pi}{3}$.  The real values of these points are $r_1\cos \left( \frac{\pi}{2}+t\frac{\pi}{3} \right)$ and $r_2\cos \left( \frac{\pi}{6}+t\frac{\pi}{3} \right)$ respectively, with $0 \leq r_1,r_2 \leq 2$.  Note for this range on $t$,  $-1 \leq \cos \left( \frac{\pi}{2}+t\frac{\pi}{3} \right) \leq 0$ and $0 \leq \cos \left( \frac{\pi}{6}+t\frac{\pi}{3} \right) \leq 1$.  This means that the endpoints of these rays, at $r_1=r_2=2$, will be the farthest left and right points of $\setS$.  The real-diameter is the difference of these two values,
\begin{equation}
\label{set_S_Width_equation}
2\cos \left( \dfrac{\pi}{6}+t\dfrac{\pi}{3} \right) - 2\cos \left( \dfrac{\pi}{2}+t\dfrac{\pi}{3} \right).
\end{equation}

Using differentiation to find the maximum on Equation \eqref{set_S_Width_equation}, we find the maximum value of the width occurs at $t=\frac{1}{2}$.  Plugging this in, we find the width to be 2, thus $\setS$ is no wider than 2 units. This means the ``real-width'' of $\setU'_k$ is at most 2 (since $\setU'_k \subset \setS$). Remembering that $Re(c)$ is the position of the minor axis of $\ellipse$, we see
$$
Re\left( c+2\sqrt{a} \right) - Re(c) = Re(c)+2\sqrt{a} - Re(c)=2\sqrt{a} \geq 2
$$
since $a\geq 1$. Therefore when $v_+$ is at its right-most point on $\partial \setU'_k$, the minor axis will be a distance of at least 2 units to the left, a distance greater than or equal to the ``real-width'' of $\setU'_k$ (see Figure \ref{fig:uprime_k_minor_axis}).

\begin{figure}
\centering
\includegraphics[scale=.35]{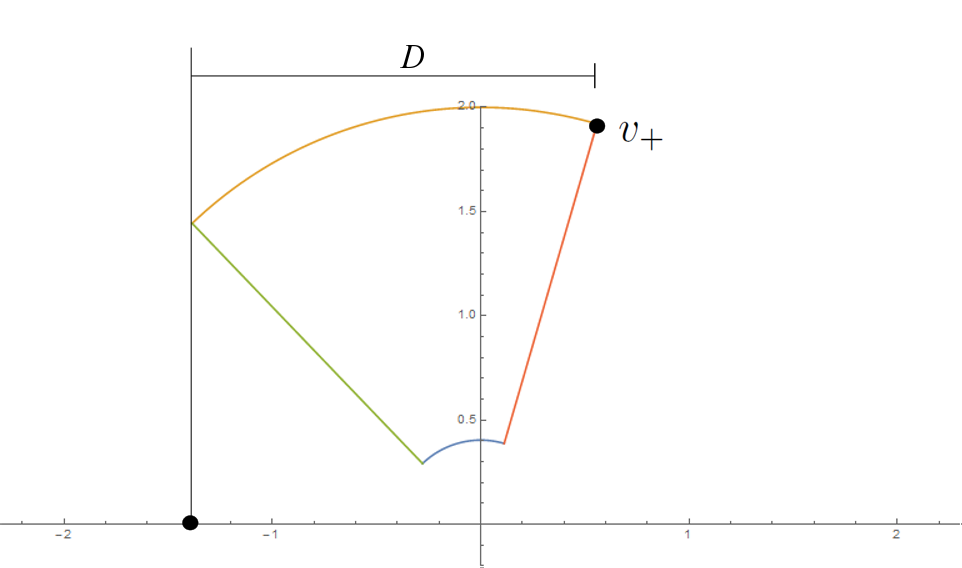}  
\caption{\label{fig:uprime_k_minor_axis} A $\setU'_{k}$ intersecting the imaginary axis.  The maximum width of $\setU'$ is ${D=2\sqrt{a}\geq 2}$ from $v_+$ to the minor axis of $\ellipse$. The vertical line, $x=Re(c)$, is the minor axis of $\ellipse$.}
\end{figure}

This means the minor axis of $\ellipse$ will not intersect $\setU'_k$. A symmetrical argument can be used for any $\setU'_k$ that intersects the negative imaginary axis, and thus $\setU'_k \subset \setU$ for all $k=0,1,...,n-1$.

\end{proof}

Knowing $\setU'_k \subset \setU$ for all ${k=0,1,...,n-1}$, we meet the requirements for a polynomial-like map.

\begin{proposition}
\label{R_polynomial_like_Each_UK_prop}
With the same hypotheses as the previous lemma, $R_{n,a,c}$ is a polynomial-like map of degree two on $\setU'_k$ when $c \in \setWck$.
\end{proposition} 


\begin{proof}
By design, each $\setU'_k$ is centered around a unique critical point of $R_{n,a,c}$, specifically $\lvert a \rvert^{\frac{1}{2n}}e^{i\frac{2\pi k}{n}}$.  $R_{n,a,c}$  is also a two-to-one map on each $\setU'_k$ by the discussion above. Additionally it is clear that $R_{n,a,c}$ is analytic on each $\setU'_k$. With this and Lemma \ref{UPrimeK_Contained_ForAllK_Lemma}, we fulfill Definition \ref{Polynomial-like_definition} and $R_{n,a,c}$ is a polynomial-like map of degree two on each $\setU'_k$.
\end{proof}

\subsection{Parameter Plane Results: $c$ Plane}

Knowing $R_{n,a,c}$ is polynomial-like of degree two on each $\setU'_k$, we next show how the hypothesis of Theorem \ref{DH_Mandelbrot_existence_Criterion_Theorem} is satisfied, in this case of $a$ fixed, and we locate multiple baby $\mandel$'s in the $c$-parameter plane.

We first pick a $k$ from $\lbrace 0,1,...,n-1 \rbrace$, and observe $\left\{ R_{n,a,c} \right\}_{c \in \setWck}$ as the family of functions in Theorem \ref{DH_Mandelbrot_existence_Criterion_Theorem}.  We have already shown each member of this family is polynomial-like of degree two on $\setU'_k$. Both $\partial \setU'_k$ and $\setU$ clearly vary analytically with $c$, as well as $R_{n,a,c}(z)$.  Last, the implicit definition of $\setWck$ makes $v_+$ take a closed loop around $\partial \setU'_k$ as $c$ loops around $\partial \setWck$.  Now we satisfy the hypotheses of Theorem \ref{DH_Mandelbrot_existence_Criterion_Theorem} and achieve our result.

\begin{theorem}
\label{v+_Mandels_exist_multiple_cplane}
Given $n \geq 5$, $1 \leq a \leq 4$, and a fixed $k \in \lbrace 0,1,...,n-1 \rbrace$, the set of $c \in \setWck$ such that the critical orbit of $v_+$ does not escape $\setU'_k$ is homeomorphic to $\mandel$.
\end{theorem}


See Figure \ref{fig:cplaneWthree} for an example. This is part of Main Theorem \ref{Main_Theorem_CPlane}.

\begin{figure}
\centering
\includegraphics[width=1\textwidth,keepaspectratio]{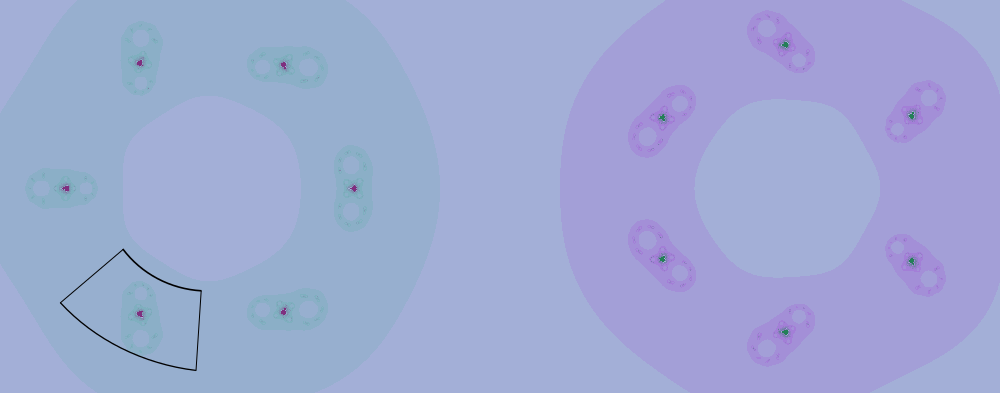}   
\caption{\label{fig:cplaneWthree} The $c$-parameter plane of $R_{n,a,c}$ and $\mathbf{W}_{n,a,k}$ with $n=6$, $a=1$, and $k=4$.}
\end{figure}

\begin{figure}
\centering
\includegraphics[width=.4\textwidth,keepaspectratio]{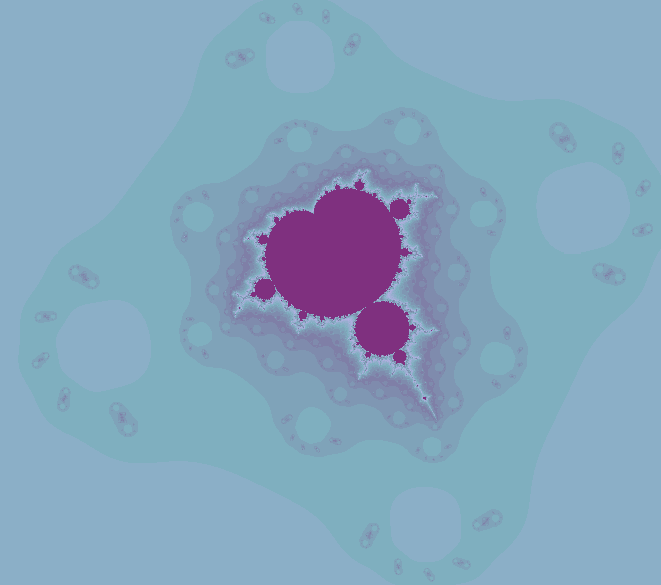}
\caption{\label{Mandel_Zoom_From_W_Figure} A baby $\mandel$ found in a zoom of Figure \ref{fig:cplaneWthree}. }
\end{figure}

Now we find more baby $\mandel$'s, but they're associated with the other critical value $v_-$.  First we set up some more symmetries present in $R_{n,a,c}$, similar to the previous case.

\begin{lemma}
\label{R_Is_Conjugate_over_Real_Axis_Lemma_c_plane}
If $a$ is real, then for every positive integer $m$,
$$
\overline{R_{n,a,\overline{c}}^m(\overline{z})} = R_{n,a,c}^m(z).
$$
\end{lemma}

\begin{proof}
Proving by induction, we first establish the base case:
\begin{eqnarray*}
\overline{R_{n,a,\overline{c}}(\overline{z})} ~=~ \overline{\overline{z}^n + \dfrac{a}{\overline{z}^n} + \overline{c}} ~=~ \overline{\overline{z^n}} + \dfrac{a}{\overline{\overline{z^n}}} + \overline{\overline{c}} ~=~ z^n + \dfrac{a}{z^n} + c ~=~ R_{n,a,c}(z).
\end{eqnarray*} 

Assuming the statement is true for $m-1$, we have:
\begin{eqnarray*}
& ~ & \overline{R_{n,a,\overline{c}}^m(\overline{z})}\\
& = & \overline{R_{n,a,\overline{c}}^{m-1}(R_{n,a,\overline{c}}(\overline{z}))}\\
& = & \overline{R_{n,a,\overline{c}}^{m-1}(\overline{R_{n,a,c}(z)})} \hspace{1cm} \text{(By the base case)}\\
& = & R_{n,a,c}^{m-1}(R_{n,a,c}(z)) \hspace{1cm} \text{(By the induction assumption)}\\
& = & R_{n,a,c}^m(z)
\end{eqnarray*}

and we have established our result by induction.
\end{proof}

Using Lemma \ref{R_Is_Conjugate_over_Real_Axis_Lemma_c_plane} on the critical orbits of $R_{n,a,c}$, we get:

\begin{lemma}
\label{cParameter_Plane_Symmetric_over_realAxis_lemma}
If $a\in \mathbb{R}$, then $\overline{R_{n,a,\overline{c}}^m(v_{\pm})} =R_{n,a,c}^m(v_{\pm})$ for all $m \in \mathbb{N}$.
\end{lemma}

\begin{proof}
Using Lemma \ref{R_Is_Conjugate_over_Real_Axis_Lemma_c_plane} and the fact that $a$ is real, we get:
\begin{eqnarray*}
\overline{R_{n,a,\overline{c}}^m(v_{\pm})} = \overline{R_{n,a,\overline{c}}^m(\overline{c} \pm 2\sqrt{a})} =  \overline{R_{n,a,\overline{c}}^m(\overline{c \pm 2\sqrt{a}})} 
\\ =  R_{n,a,c}^m(c \pm 2\sqrt{a}) = R_{n,a,c}^m(v_{\pm}).
\end{eqnarray*}
\end{proof}

Lemma \ref{cParameter_Plane_Symmetric_over_realAxis_lemma} yields that the dynamics of the critical orbits above the real axis of the $c$-plane will be the same as the dynamics below. Given these above lemmas we now see that for $n$ odd and $a$ real, the $c$-parameter plane of $R_{n,a,c}$ is symmetric over the real-axis and through the origin.  Combining these two symmetries yields the next Lemma.

\begin{lemma}
\label{cParameter_Plane_Symmetric_over_imaginaryAxis_lemma}
For $n \geq 3$, odd, and $a$ real, the boundedness locus in the $c$-parameter plane is symmetric across the imaginary axis in the $c$-plane.
\end{lemma}

We can now finish Main Theorem \ref{Main_Theorem_CPlane}-(i). 


\begin{theorem}
\label{v-_Mandels_exist_multiple_through_symmetry_theorem}
If $n \geq 5$, odd, and $1 \leq a \leq 4$, then for each baby $\mandel$ associated with the critical orbit of $v_+$, there exists a matching baby $\mandel$ associated with the critical orbit of $v_-$.  Each $v_-$ baby $\mandel$ is a reflection of a $v_+$ baby $\mandel$ over the imaginary axis of the $c$-plane.
\end{theorem}

\begin{proof}
This is a result of Theorem \ref{v+_Mandels_exist_multiple_cplane} and Lemma \ref{cParameter_Plane_Symmetric_over_imaginaryAxis_lemma}. Therefore there are $2n$ baby $\mandel$'s in the $c$-plane under these restrictions of $n$ and $a$.
\end{proof}

Now we have shown all the black figures we see in Figure \ref{fig:cplaneWthree} are indeed homeomorphic to $\mandel$.

\section{Extending Results: the case of small $a$}
\label{Extend_Results_Section}

Having located $2n$ baby $\mandel$'s, we now look to push the range of fixed parameter values in which they exist, toward smaller values of $a$, approaching the degenerate $a=0$ case.  We shall find baby $\mandel$'s in the $c$-plane for $a$ as small as $\frac{1}{10}$, but this requires increasing the minimum bound on the degree $n$. One could push $a$ even smaller, but then raising $n$ would be necessary. A potential direction for future work would be to study the needed lower bound on $n$ as $a$ decreases to $0$.

\subsection{Dynamical Plane Results}

Smaller values of $a$ force us to decrease our escape radius of $R_{n,a,c}$, as well as further restrict the degree of the rational functions. We shall restrict the argument of $v_+$ to a small interval around $0$, centering our domain around the positive real axis, the same as $\setU'_0$ from before. 

First some notation.

\begin{definition}
Let $ A^*(\infty)$ denote the Basin of Attraction of Infinity, also called the escape locus. That is, $z\in A^*(\infty)$ iff the orbit of $z$ under $R_{n,a,c}$ escapes to $\infty$. 
\end{definition}

We will prove for $\frac{1}{10} \leq \lvert a \rvert \leq 1$ that any point outside of a modulus of 1.25 will escape to $\infty$ under $R_{n,a,c}$ (instead of using 2 as before).

Given this escape radius and $a$ at its maximum, having $\lvert c \rvert \geq 3.25$ will guarantee that $v_+ \in A^*(\infty)$ (i.e. $v_+$ lies outside the escape radius).

\begin{lemma}
\label{TighterEscapeRadius}
If $n \geq 11$, $\lvert c \rvert \leq 3.25$, $\frac{1}{10} \leq a \leq 1$, then any $z$ such that $\lvert z \rvert > \frac{5}{4}$ will lie in $A^*(\infty)$ of $R_{n,a,c}$.
\end{lemma}

\begin{proof}
Once again using the results in \cite{boydschul}, given any $\epsilon > 0$ and for $n$ sufficiently large, the filled Julia set of $R_{n,a,c}$ is contained in $\mathbb{D} \left( 0,1+\epsilon \right)$. Anything outside the radius of  $1+\epsilon$ escapes to $\infty$.

Similar to Lemmas \ref{EscapeRadius_lemma_aplaneCase} and \ref{EscapeRadius_lemma_cplaneCase}, if $N$ satisfies $(1+\epsilon)^N > 3\text{Max} \lbrace 1,\lvert a \rvert, \lvert c \rvert \rbrace$, then for $n \geq N$ the orbits of values $\lvert z \rvert > 1+\epsilon$ must tend to $\infty$.  Here, we have $\epsilon=0.25$, and $3\text{Max} \lbrace 1,\lvert a \rvert, \lvert c \rvert \rbrace = 9.75$ for $a$ and $c$ at their greatest moduli. So when we solve this equation for $N$, we find $N > \dfrac{\ln(9.75)}{\ln(1.25)} \approx 10.2$, thus $n \geq 11$ will satisfy the criterion.
\end{proof}

Combining this new escape criterion with Lemma \ref{involution_prop} yields that the orbit of any ${\lvert z \rvert < \frac{4}{5}a^{\frac{1}{n}}}$ will tend to $\infty$, thus we get the following lemma.

\begin{lemma}
Under the same hypothesis as Lemma \ref{TighterEscapeRadius}, the filled Julia set of $R_{n,a,c}$ is contained in the annulus $\mathbb{A} \left( \frac{4}{5}a^{\frac{1}{n}}, \frac{5}{4} \right)$.
\end{lemma}

With this new restriction on the location of the filled Julia set of $R_{n,a,c}$, we define a set $\setUp'$ that takes on the same role as that of $\setU'_0$ from earlier.  Remember that $\Arg(v_+)$ is restricted to a neighborhood of zero. We prove $R_{n,a,c}$ is polynomial-like of degree two on the set
\begin{equation}
\boxed{
\setUp' =\left\{  z=re^{i\theta} \   \middle|  \ \ \dfrac{4}{5}a^{\frac{1}{n}}<~r~<\dfrac{5}{4}~~and~~\dfrac{\psi-\pi}{2n}<~\theta~<\dfrac{\psi+\pi}{2n} \right\},}
\end{equation}
where $\psi = \Arg(a) = 0$ since $a \in \mathbb{R}^+$. We also define $\setUp = R_{n,a,c}(\setUp')$. The critical point $a^{\frac{1}{2n}}$ is contained in this new $\setUp'$ and is mapped to $v_+$. With this change to the inner and outer boundaries of $\setUp'$, the image under $R_{n,a,c}$ is still half an ellipse cut by the minor axis and centered at $c$ but now has 
$$
\text{semi-major~axis~length}: ~ \left( \dfrac{5}{4} \right)^n + {\left( \dfrac{4}{5} \right)^n}\lvert a \rvert
$$
$$ 
\text{semi-minor~axis~length}: ~ \left( \dfrac{5}{4} \right)^n - {\left( \dfrac{4}{5} \right)^n}\lvert a \rvert
$$
We refer to this new ellipse as $\newellipse$ and can use a different representation:

\begin{equation}
\label{New_Ellipse_Equation_Definition}
\boxed{
\newellipse = \left\{ \begin{array}{lcl}
& x= & \left( \left( \dfrac{5}{4} \right)^n + {\left( \dfrac{4}{5} \right)^n}\lvert a \rvert \right)\cos(n\theta) \\
& y= & \left(\left( \dfrac{5}{4} \right)^n - {\left( \dfrac{4}{5} \right)^n}\lvert a \rvert\right)\sin(n\theta)
\end{array} \right\}
}~.
\end{equation}

Similar to before, we must show that $R_{n,a,c}$ is polynomial-like of degree two on $\setUp'$.  First we will show $\setUp'$ is contained in $\newellipse$, and then show further containment of $\setUp'$ inside $\setUp$.
\begin{lemma}
\label{tighter_radius_Uprime_in_ellipse_lemma}
$\setUp' \subset \newellipse$ for $n \geq 11$, $\frac{1}{10} \leq a \leq 1$, and $c$ such that $\lvert v_+ \rvert \leq \frac{5}{4}$.
\end{lemma}

\begin{proof}
Similar to the proof of Lemma \ref{Uprime_contained_in_ellipse_afixed_case_lemma}, we shall define the ellipse by its alternate definition:
\begin{equation}
\left\{  z \   \middle| ~ \lvert z-(c-\sqrt{a}) \rvert + \lvert z-(c+\sqrt{a}) \rvert \leq 2\left( \left( \dfrac{5}{4} \right)^n + {\left( \dfrac{4}{5} \right)^n}\lvert a \rvert \right) \right\}.
\end{equation}
Since $\setUp' \subset \mathbb{D}(0,\frac{5}{4})$, showing $\mathbb{D}(0,\frac{5}{4})$ is contained in the ellipse will suffice.  The largest values of $\lvert v_{\pm} \rvert$  will be observed.  We start with $\lvert v_+ \rvert = \frac{5}{4}$ on the outer boundary and find the largest possible $\lvert v_- \rvert$:
\begin{eqnarray*}
\lvert v_+ \rvert &=& \lvert c + 2\sqrt{a} \rvert = \dfrac{5}{4} ~ \Rightarrow ~ c= \dfrac{5}{4}e^{i\theta} - 2\sqrt{a}\\
&\Rightarrow & ~ \lvert v_- \rvert = \lvert c - 2\sqrt{a} \rvert = \left| \dfrac{5}{4}e^{i\theta} - 4\sqrt{a} \right| \\
\end{eqnarray*}
for some $\theta \in [0,2\pi)$. Using derivatives we find the maximum of $\lvert v_- \rvert$ occurs at $\theta=\pi$, so the maximum $\lvert v_- \rvert$ is $\frac{5}{4} + 4\sqrt{a}$, and we have bounds on the foci.

Now we have:
\begin{eqnarray*}
& &\lvert z-(c-\sqrt{a}) \rvert + \lvert z-(c+\sqrt{a}) \rvert\\
& \leq & 2\lvert z \rvert + \lvert c+2\sqrt{a} \rvert + \lvert c-2\sqrt{a} \rvert \\
& \leq & 2\left(\dfrac{5}{4}\right) + \left(\dfrac{5}{4}\right) + \left(\dfrac{5}{4} + 4\sqrt{a}\right)\\
& = & 5 + 4\sqrt{a} \\
& \leq &  5+ 4  \hspace{1cm} (\text{Since}~a \leq 1)\\
& \leq & 2\left( \dfrac{5}{4} \right)^{11} \leq ~ 2\left( \left( \dfrac{5}{4} \right)^n + {\left( \dfrac{4}{5} \right)^n}\lvert a \rvert \right). \\ 
\end{eqnarray*}

Thus the image of any point in $\setUp' \subset \mathbb{D}(0,\frac{5}{4})$ will be contained in this ellipse.

\end{proof}

With $\setUp'$ contained in the ellipse, we just need to show the minor axis of the ellipse does not intersect $\setUp'$.  This will guarantee that $\setUp' \subset \setUp$.

\begin{lemma}
\label{tighter_radius_uprime_in_U}
For $n \geq 11$, $\frac{1}{10} \leq a \leq 1$, and $c$ such that $\lvert v_+ \rvert \leq \frac{5}{4}$, $\setUp' \subset \setUp$.
\end{lemma}

\begin{proof}
$\setUp$ is the half of the ellipse that contains $v_+$ which is the right-half in this case. We just have to show the minor axis of the ellipse does not intersect $\setUp'$ and in fact lies to the left to $\setUp'$.

The minor axis is a straight vertical line centered at $c$.  Its horizontal position is $Re(c)$, a value that depends on $v_+$.  At its greatest modulus we observe $v_+$ on the outer arc of $\setUp'$, so $c+2\sqrt{a} = \frac{5}{4}e^{i\theta}$ for $\lvert \theta \rvert \leq \frac{\pi}{2n}$, thus 
$$
Re(c) = \dfrac{5}{4}\cos(\theta) - 2\sqrt{a}~\leq~\dfrac{5}{4} - 2\sqrt{a}.
$$

Now we show the minor axis lies to the left of the inner arc of $\setUp'$. This happens if 
$$
Re(c) < Re \left( \dfrac{4}{5}a^{\frac{1}{n}}e^{i\theta} \right),~i.e.~ \dfrac{5}{4} - 2\sqrt{a} < \dfrac{4}{5}a^{\frac{1}{n}}\cos(\theta).
$$ 
Since $n \geq 11$, $\cos(\theta)$ will be minimal at $\theta = \pm \frac{\pi}{2n}$ and thus the minimum value of $Re \left( \frac{4}{5}a^{\frac{1}{n}}e^{i\theta} \right)$ will be $\frac{4}{5}a^{\frac{1}{n}}\cos\left( \dfrac{\pi}{2n} \right)$ with respect to $\theta$. To minimize further, we take a derivative with respect to $n$ and find 
$$ \dfrac{\partial}{\partial_{n}}\left( \dfrac{4}{5}a^{\frac{1}{n}}\cos\left( \dfrac{\pi}{2n} \right) \right) = \dfrac{2a^{1/n}\left( \pi\sin \left( \dfrac{\pi}{2n} \right) - 2\cos\left( \dfrac{\pi}{2n} \right)\ln(a) \right)}{5n^2} > 0
$$ 
since $a \leq 1$. Therefore the derivative is positive and the horizontal position of the inner arc of $\setUp'$ increases as n increases. As $\frac{4}{5}a^{\frac{1}{n}}\cos\left( \frac{\pi}{2n} \right)$ increases with $a$ as well, the minimum value of this equation occurs at the minimal values of $n=11$ and $a= \frac{1}{10}$. Thus
$$ 
Re \left( \dfrac{4}{5}a^{\frac{1}{n}}e^{i\theta} \right) \geq \dfrac{4}{5}\left(\frac{1}{10}\right)^{\frac{1}{11}}\cos\left( \dfrac{\pi}{22} \right) \approx 0.6423.
$$
At this point the minor axis will be located at $\frac{5}{4}-2\sqrt{\frac{1}{10}} \approx 0.618$, and we have that the minor axis lies to the left of $\setUp'$.  The horizontal position of the minor axis does not depend on $n$, and moves to the left as $a$ increases.  With this, it is assured that the minor axis of the ellipse will never pass through $\setUp'$ and $\setUp' \subset \setUp$.
\end{proof}

\begin{proposition}
\label{R_polynomial_Like_on_tighter_UPrime}
Under the same assumptions as Lemma \ref{tighter_radius_uprime_in_U}, $R_{n,a,c}: \setUp' \rightarrow \setUp$ is a polynomial-like map of degree two.
\end{proposition}

\begin{proof}
$R_{n,a,c}$ is analytic on $\setUp'$ by choice of the boundaries, as well as two-to-one with a single critical point. Thus, $R_{n,a,c}$ satisfies the definition of a polynomial-like map of degree two by Lemma \ref{tighter_radius_uprime_in_U}.
\end{proof}

\subsection{Parameter Plane Results}

Now that we have a family of degree two polynomial-like maps, we define a $\setWp$ in this case by 
\begin{equation}
\boxed{
\setWp = \left\{  c \   \middle|  \ \ \dfrac{4}{5}a^{1/n} \leq \left| v_+ \right| \leq \dfrac{5}{4} ~~ and ~~ \dfrac{-\pi}{2n} \leq \Arg(v_+) \leq \dfrac{\pi}{2n} \right\}}
\label{eqn:defnWp}
\end{equation}
to invoke Theorem \ref{DH_Mandelbrot_existence_Criterion_Theorem} and follow the same criteria to show that a baby $\mandel$ is contained in this $\setWp$.


\begin{theorem}
\label{Tighter_radius_mandel_exists_in_cPlane_Theorem}
For $\frac{1}{10} \leq a \leq 1$ and $n \geq 11$, the set of $c$-values contained in $\setWp$ such that the critical orbit of $v_+$ does not escape $\setUp'$ is homeomorphic to $\mandel$.
\end{theorem}

\begin{proof}
By design of $\setWp$ it is clear that as $c$ loops around $\partial \setWp$, we have that $v_+$ will make a loop around $\partial \setUp' \subset \setUp - \setUp'$.

$R_{n,a,c}$ is polynomial-like of degree two on $\setUp'$ by Proposition \ref{R_polynomial_Like_on_tighter_UPrime} and thus we have satisfied the criteria of Theorem \ref{DH_Mandelbrot_existence_Criterion_Theorem}. Thus there exists a baby $\mandel$ in $\setWp$.
\end{proof}

We have now extended the interval of $a$-values in which a baby $\mandel$ associated with $v_+$ exists in the $c$-plane, but restricted on the degree of $R_{n,a,c}$. 





Because of the existing symmetries in the $c$-plane, we get a baby $\mandel$ associated with $v_-$ under these criteria on $a$ and $n$ as well.

\begin{corollary}
\label{v-_mandel_exists_in_cPlane_all_aValues_corollary}
For $n \geq 11$, odd, and $\frac{1}{10} \leq a \leq 1$, there exists a baby $\mandel$ associated with $v_-$ lying in the $c$-parameter plane of $R_{n,a,c}$, within the reflection over the imaginary axis of the set $\setWp$.
\end{corollary}

\begin{proof}
This follows from the existence of a baby $\mandel$ associated with $v_+$ in 
Theorem~\ref{Tighter_radius_mandel_exists_in_cPlane_Theorem}, 
as well as the symmetry given in Lemma \ref{cParameter_Plane_Symmetric_over_imaginaryAxis_lemma}.  Therefore the baby $\mandel$ associated with $v_-$ is a reflection of the baby $\mandel$ associated with $v_+$ from 
Theorem~\ref{Tighter_radius_mandel_exists_in_cPlane_Theorem}, over the imaginary axis of the $c$-parameter plane of $R_{n,a,c}$.
\end{proof}

In Figure \ref{CPlane_n11_a022_With_Zoom_2Mandels_Figure} we can see an example of these baby $\mandel$'s existing under these new criteria.  A zoom in is necessary as the baby $\mandel$'s shrink as $n$ grows.  The green baby $\mandel$ represents the orbit of $v_-$ and the purple baby $\mandel$ represents the orbit of $v_+$.

\begin{figure}
\centering
\includegraphics[scale=0.7]{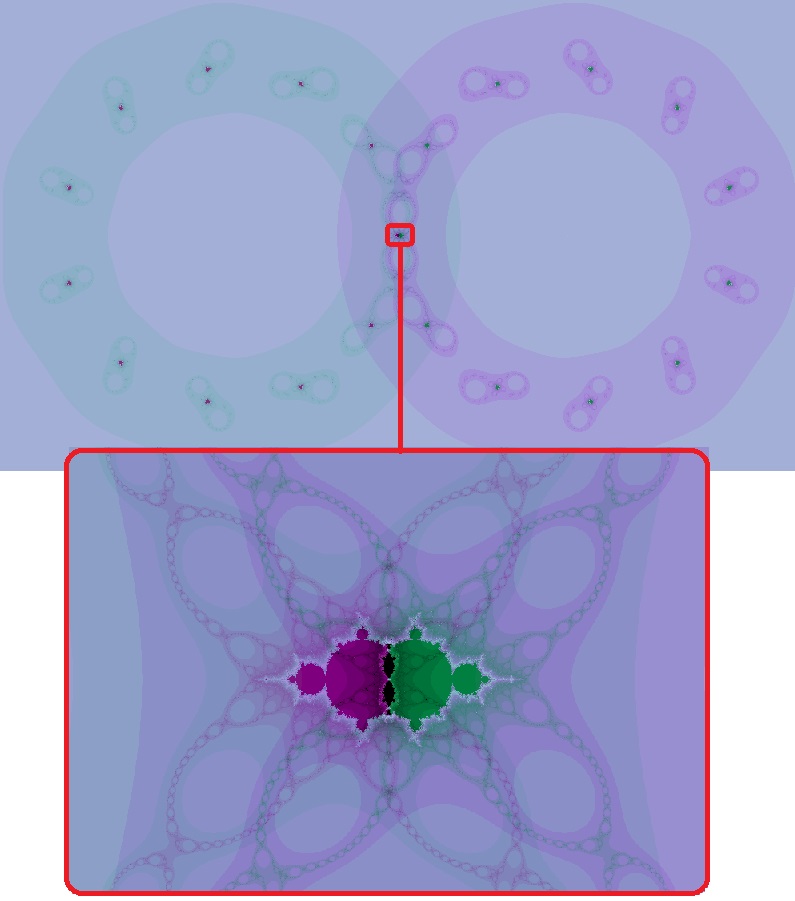}
\caption{\label{CPlane_n11_a022_With_Zoom_2Mandels_Figure} The $c$-parameter plane of $R_{n,a,c}$ for $n=11$ and $a=0.22$, and a zoom into the origin.  The green $\mandel$ is associated with $v_-$ and the purple $\mandel$ is associated with $v_+$.}
\end{figure}

\subsection{Mandelbrots Passing through each other}

Now we will take a closer look at two baby $\mandel$'s which intersect in the $c$-plane and pass through one another along a line of $a$ values.

We define the center of a baby $\mandel$ in the $c$-parameter plane to be the $c$-value for which the critical point is the same as the critical value; i.e., the critical point is a fixed point of $R_{n,a,c}$.  If we wish to find the center of a baby $\mandel$ associated with $v_+$, then we are solving the equation $a^{1/2n} = c + 2\sqrt{a}$ for $c$, which is 
$$
c_+ = a^{1/2n} - 2\sqrt{a}.
$$
By Lemma \ref{cParameter_Plane_Symmetric_over_imaginaryAxis_lemma},  the baby $\mandel$ associated with $v_-$ is just a reflection over the imaginary axis of the $c$-parameter plane.  We reflect $c_+$ over the imaginary axis to find the center of the baby $\mandel$ associated with $v_-$ to be
$$
c_- = 2\sqrt{a} - a^{1/2n}.
$$
Using this, we can give an exact case when the two baby $\mandel$'s overlap.

\begin{proposition}
\label{Mandels_With_Same_Center_Proposition}
For $n \geq 11$, odd, and $a = \left( \frac{1}{4} \right)^{\frac{n}{n-1}}$, two baby $\mandel$'s in the $c$-parameter plane associated with $v_+$ and $v_-$ respectively overlap and have the same center.  Further, this same center is the origin of the $c$-plane.
\end{proposition}

\begin{proof}
We first solve the equation $c_+ = c_-:$ 
$$
 a^{1/2n} - 2\sqrt{a} = 2\sqrt{a} - a^{1/2n} \Rightarrow 2\sqrt{a} = a^{1/2n} \Rightarrow a = 2^{\frac{1}{\frac{1}{2n} - \frac{1}{2}}} = \left( \dfrac{1}{4} \right)^{\dfrac{n}{n-1}}
$$
and find when we plug this value into $c_+$:
\begin{eqnarray*}
&~& \left(\left( \dfrac{1}{4} \right)^{\dfrac{n}{n-1}} \right)^{1/2n}-2\sqrt{\left( \dfrac{1}{4} \right)^{\dfrac{n}{n-1}}}\\ 
&=& \left( \dfrac{1}{2} \right)^{\dfrac{1}{n-1}}
-\left(\dfrac{1}{2} \right)^{\dfrac{n}{n-1}-1}\\
&=& \left( \dfrac{1}{2} \right)^{\dfrac{1}{n-1}} - \left(\dfrac{1}{2} \right)^{\dfrac{1}{n-1}}\\
&=& 0.
\end{eqnarray*}
\end{proof}

For any odd $n \geq 11$ we are thus given an $a$-value for which two baby $\mandel$'s of opposite critical orbits will intersect at the origin and have the same center (see Figure \ref{Overlapping_Mandels_Same_Center_n11} for an example).  Knowing that these baby $\mandel$'s intersect allows us to find an interval of $a$-values, such that as $a$ varies continuously from one end to the other, the baby $\mandel$ sets will pass completely through each other.

\begin{figure}
\centering
\includegraphics[width=0.75\textwidth,keepaspectratio]{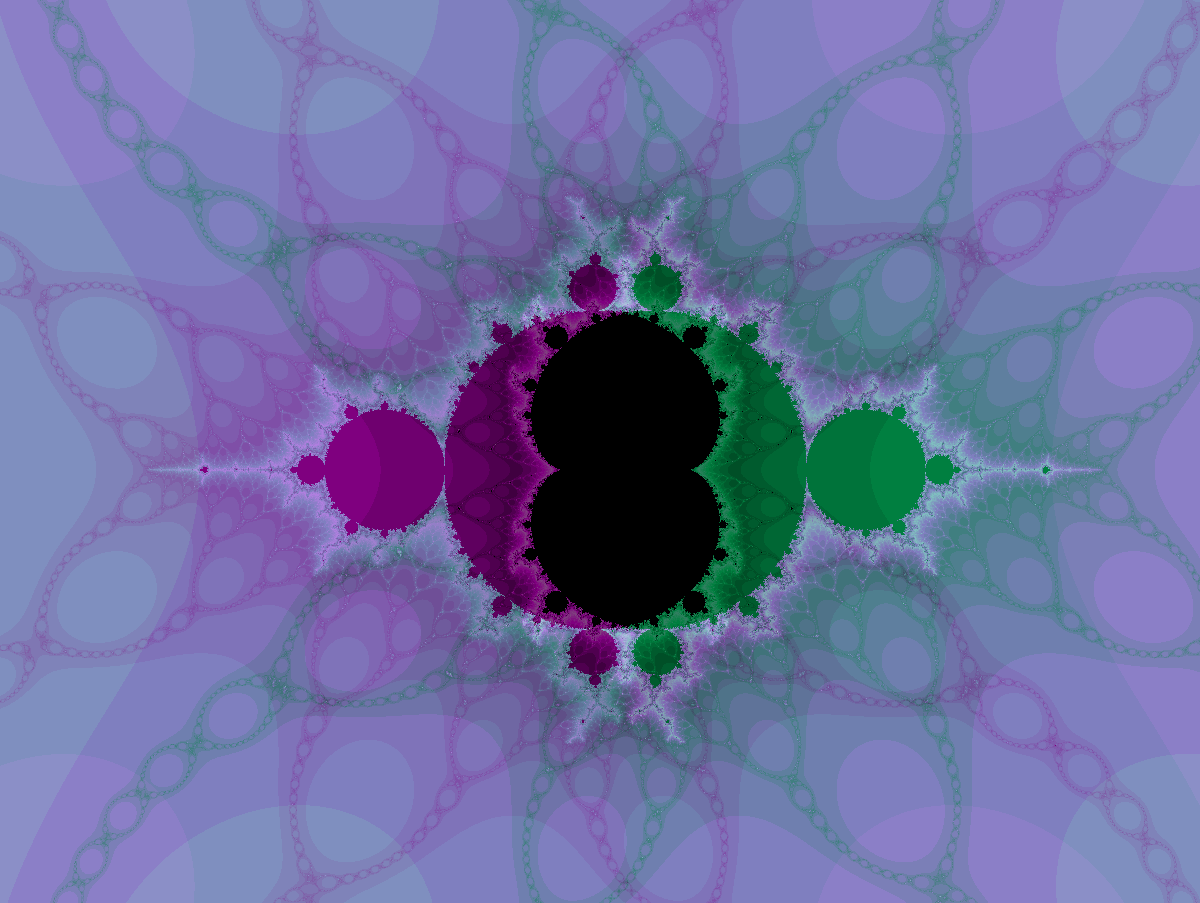}
\caption{\label{Overlapping_Mandels_Same_Center_n11} A zoom in on the $c$-parameter plane of $R_{n,a,c}$ for $n=11$ and $a=0.2176$, here two baby $\mandel$'s are both centered at the origin.}
\end{figure}

\begin{lemma}
\label{Mandels_intersect_and_lie_completely_on_real_axis_lemma}
For any odd $n \geq 11$, and $\frac{1}{10} \leq a \leq 4$, two baby $\mandel$'s associated with $v_-$ and $v_+$ are centered at the real axis of the $c$-plane, and move along it continuously as $a$ varies continuously.
\end{lemma}

\begin{proof}
In Lemma \ref{cParameter_Plane_Symmetric_over_realAxis_lemma}, we showed these baby $\mandel$'s are symmetric about the real axis when $a$ is positive and real. Thus, as $a$ varies along the real axis, the baby $\mandel$'s vary, keeping their center and only axis of symmetry in the real axis. 
\end{proof}

Knowing that the centers of these two baby $\mandel$'s stay on the real axis as they change position, we can prove that the two sets actually pass completely through one another as $a$ changes continuously in the range of interest.

\begin{proposition}
Letting $n \geq 11$ and odd, as $a$ increases from $\frac{1}{10}$ to $1$, two baby $\mandel$'s associated with $v_-$ and $v_+$ move along the real axis of the $c$-parameter plane of $R_{n,a,c}$ in opposite directions and completely pass through each other.
\end{proposition}

\begin{proof}
By Theorem \ref{Tighter_radius_mandel_exists_in_cPlane_Theorem} the baby $\mandel$ associated with $v_+$ lies in $\setWp$ with the tighter radius, this means that
$$
\dfrac{4}{5}a^{1/n} \leq \lvert c + 2\sqrt{a} \rvert \leq \dfrac{5}{4} ~ \Rightarrow ~ Re(c) \in \left[\dfrac{4}{5}a^{1/n} - 2\sqrt{a}~,~\dfrac{5}{4}-2\sqrt{a}\right].
$$
We refer to this interval as $\mathbf{I}_1$ and call the respective endpoints $\omega_1$ and $\omega_2$.  So, any real $c$ value in this baby $\mandel$ must lie in $\mathbf{I}_1$.  Because the other baby $\mandel$ associated with $v_-$ is a reflection of the first over the imaginary axis, the real values that lie in this baby $\mandel$ are just a reflection of $\mathbf{I}_1$ across the imaginary axis and so the real values of the baby $\mandel$ associated with $v_-$ lie in
$$
\mathbf{I}_2 = \left[-\omega_2~,-\omega_1 \right] = \left[2\sqrt{a} - \dfrac{5}{4}~,~2\sqrt{a} - \dfrac{4}{5}a^{1/n} \right].
$$

Now both of these intervals are well defined as long as $a < \left(\frac{25}{16} \right)^n$ which is true here as $a \leq 1$. Now as we start at $a=\frac{1}{10}$ and $n=11$, $\omega_1 > 0$ and increases with $n$ so $-\omega_1 < \omega_1$ for $a=\frac{1}{10}$ and all $n \geq 11$ and therefore $\mathbf{I}_2$ lies to the left of $\mathbf{I}_1$.  As $a$ increases $\omega_1$ and $\omega_2$ decrease in value, which conversely means that $-\omega_1$ and $-\omega_2$ increase.  Therefore as $a$ increases, $\mathbf{I}_1$ will move to the left as $\mathbf{I}_2$ moves to the right.

Now at $a=1$, $\omega_2 = -\frac{3}{4}$ which is less than $-\omega_2$.  Since $\omega_2$ does not depend on $n$, then for all $n \geq 11$ and $a=1$, $\mathbf{I}_1$ lies to the left of $\mathbf{I}_2$.  Therefore the two intervals have passed through each other as they are both intervals of real values.

We know that they have to intersect at at least one point since Proposition \ref{Mandels_With_Same_Center_Proposition} gives an $a$ value for which both baby $\mandel$'s have the same center.  Lemma \ref{Mandels_intersect_and_lie_completely_on_real_axis_lemma} shows us that the two baby $\mandel$'s had to have passed through each other since their centers never left the real axis.
\end{proof}

Figure \ref{Phases_of_mandels_intersecting_figure} illustrates some different phases of the baby $\mandel$'s passing through each other in $c$-plane slices.

\begin{figure}
\centering
\includegraphics[width=0.95\textwidth,keepaspectratio]{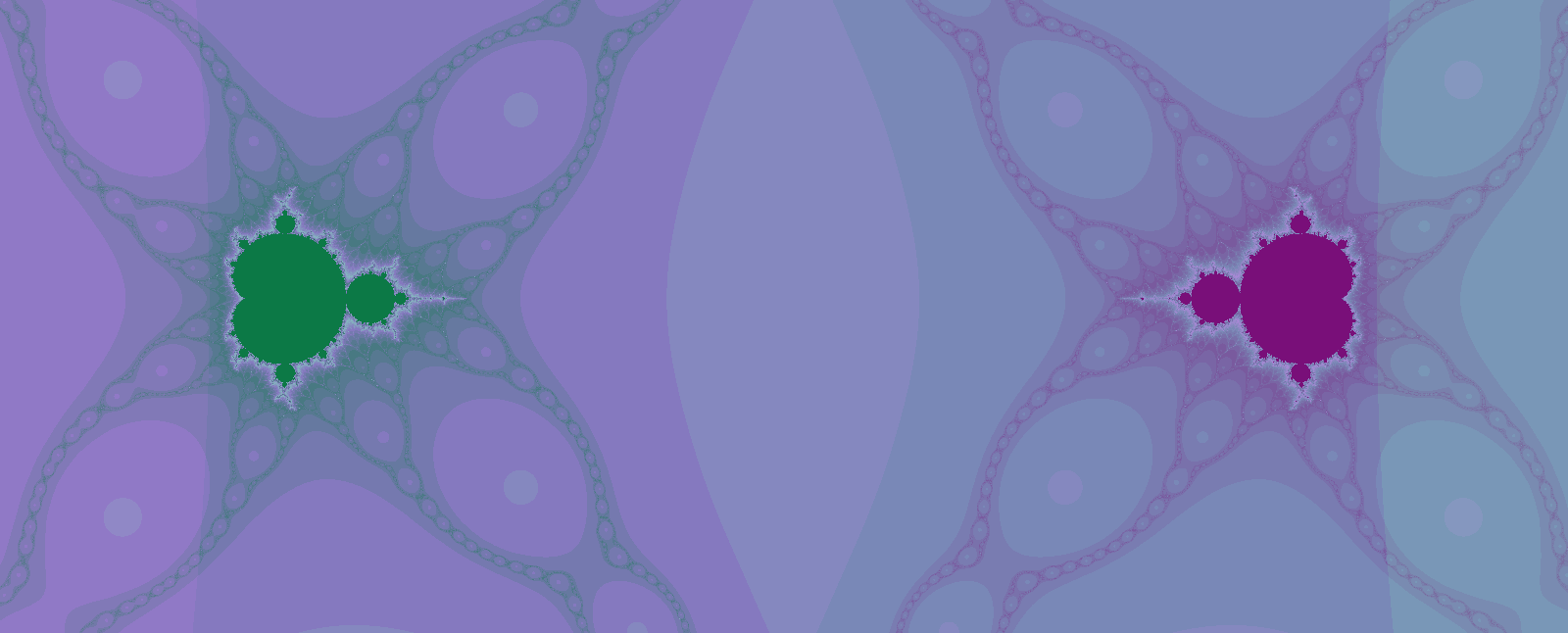} 

\includegraphics[width=0.95\textwidth,keepaspectratio]{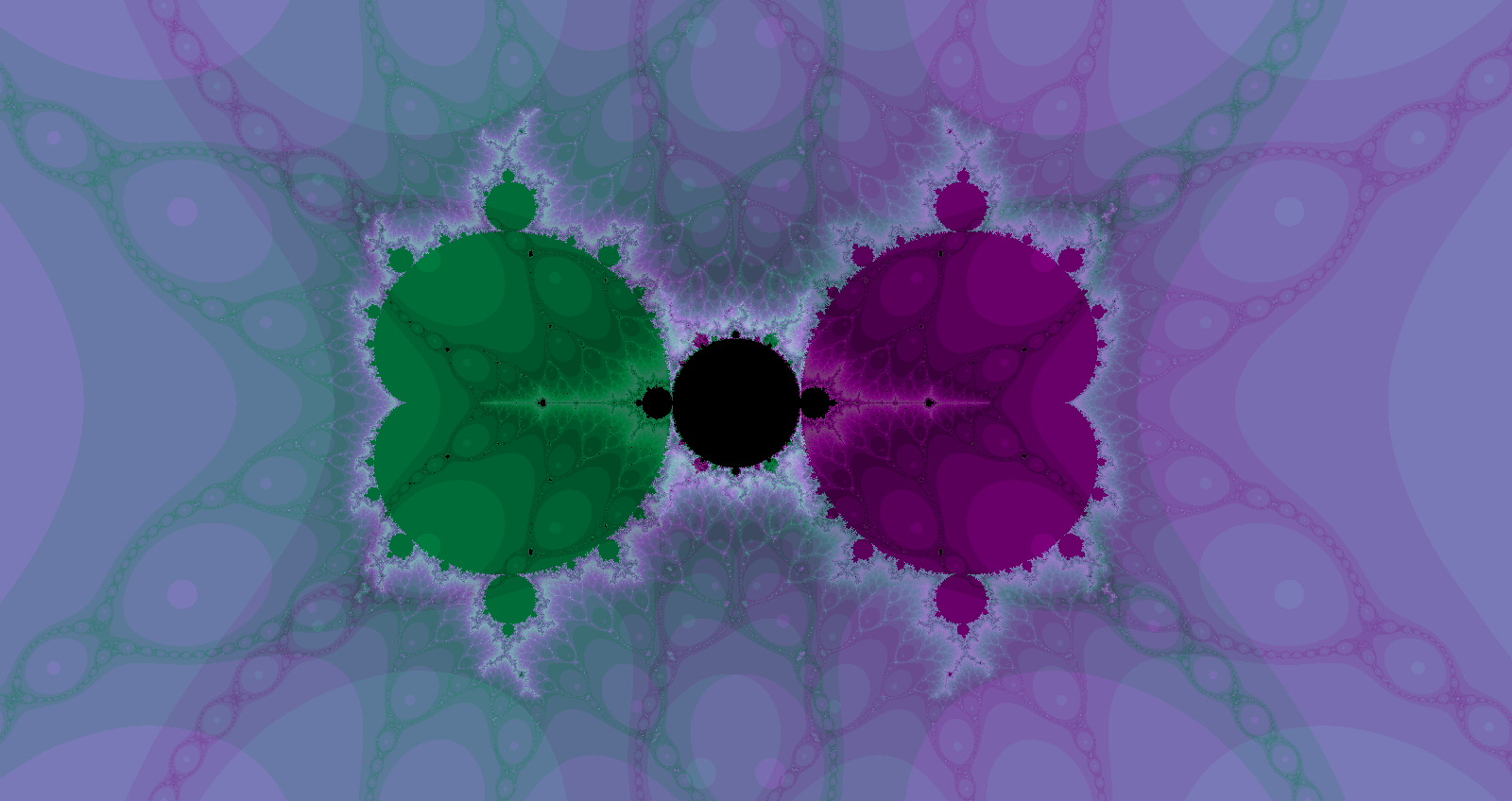} 
 
\includegraphics[width=0.95\textwidth,keepaspectratio]{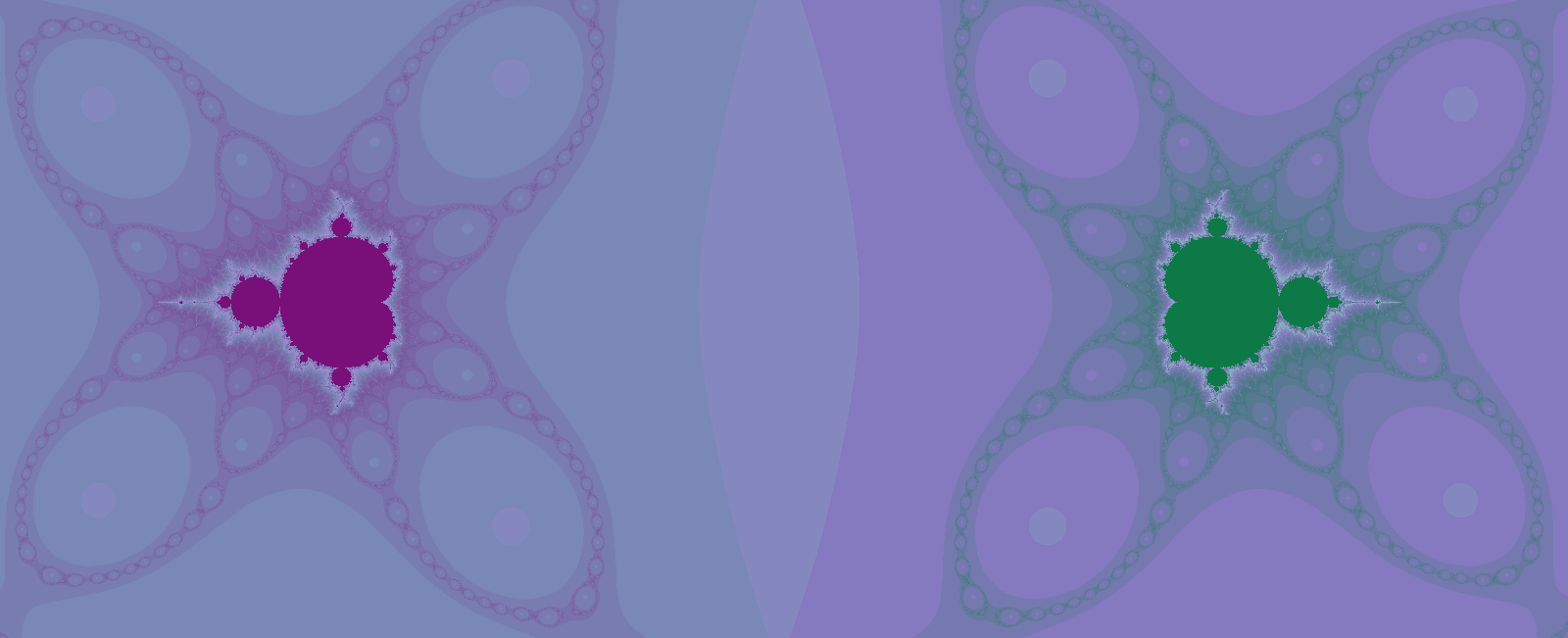} 
\caption{\label{Phases_of_mandels_intersecting_figure} Three pictures of the $c$-parameter plane of $R_{n,a,c}$, for $n=11$, showing various phases of baby $\mandel$'s intersecting (top: $a=0.175$, middle: $a=0.2098,$ bottom: $a=0.25$).}
\end{figure}


\bibliographystyle{alpha}
\bibliography{BoydMitchell_2023}

\end{document}